\documentclass[12pt,a4paper,twoside]{amsart}
\usepackage{a4wide,times, amsmath,amsbsy,amsfonts,amssymb,
stmaryrd,mathrsfs,graphicx, amscd, tikz-cd, cancel}

\usepackage{dsfont}
\usepackage[all]{xy}
\usepackage{xypic}
\usepackage{yhmath}
\usepackage{mathrsfs}
\usepackage{graphicx}
\usepackage{fancyhdr}
\usepackage{lastpage}  
\usepackage[normalem]{ulem}
\usetikzlibrary{matrix,arrows,decorations.pathmorphing}
\usepackage[active]{srcltx}
\usepackage[
	hypertexnames=false,
	hyperindex,
	pagebackref,
	pdftex,
	breaklinks=true,
	bookmarks=true,
	colorlinks,
	linkcolor=blue,
	citecolor=red,
	urlcolor=red,
]{hyperref}
\usepackage[mathscr]{eucal}

\newcommand\GL{\operatorname{GL}}
\newcommand\SL{\operatorname{SL}}

\newcommand\SU{\operatorname{SU}}
\newcommand\PU{\operatorname{PU}}
\newcommand\PSU{\operatorname{PSU}}
\newcommand\PSL{\operatorname{PSL}}

\newcommand\FFF{\mu_3}

\newcommand\HH{\mathds{H}}

\newcommand\BB{\mathds{B}}
\newcommand\CC{\mathds{C}}
\newcommand\ZZ{\mathds{Z}}
\newcommand\MM{\mathds{M}}
\newcommand\PP{\mathds{P}}

\newcommand\Res{\operatorname{Res}}
\newcommand\Mod{\operatorname{Mod}}
\newcommand\Aut{\operatorname{Aut}}
\newcommand\End{\operatorname{End}}
\newcommand\Out{\operatorname{Out}}
\newcommand\Sym{\operatorname{Sym}}

\newcommand\Hom{\operatorname{Hom}}
\newcommand\Inn{\operatorname{Inn}}
\newcommand\Ord{\operatorname{Ord}}

\newcommand\Bl{\operatorname{Bl}}

\newcommand\Acal{\mathcal A}
\newcommand\Ocal{\mathcal O}
\newcommand\Scal{\mathcal S}
\newcommand\Mcal{\mathcal M}

\newcommand\Is{\mathscr I}
\newcommand\Bs{\mathscr B}
\newcommand\Cs{\mathscr C}
\newcommand\Ds{\mathscr D}
\newcommand\Ps{\mathscr P}
\newcommand\Us{\mathscr U}
\newcommand\Os{\mathscr O}

\newcommand\Ws{\mathscr W}
\newcommand\Xs{\mathscr X}

\newcommand\Mk{\mathfrak{M}}

\newtheorem{theorem}{Theorem}[section]
\newtheorem{lem}[theorem]{Lemma}  
\newtheorem{thm}[theorem]{Theorem}  
\newtheorem{cor}[theorem]{Corollary}  
\newtheorem{prop}[theorem]{Proposition}

\theoremstyle{definition}
\newtheorem{rmk}[theorem]{Remark}

\begin{document}
\title{Geometry of the Winger Pencil}  
\author[*]{Yunpeng Zi}
\date{}  
\maketitle

\begin{abstract}
We investigate the moduli of genus 10 curves that are endowed with a faithful action of the icosahedral group $\mathcal{A}_5$. We show among other things that this has the structure of a Deligne-Mumford stack whose underlying coarse moduli space essentially consists  of two copies of the pencil of plane sextics that was introduced by Winger in 1924,  but with the unique unstable member (a triple conic) replaced by a smooth non-planar curve.
The orbifold defined by any member has genus zero and comes with 4 orbifold points. We
show that by numbering the  points, we get a fine moduli space whose base is naturally a finite cover of $\Mcal_{0,4}$.
\end{abstract}

\section{Introduction}
\label{intro}
The invariants of the  icosahedral group $\Acal_5$ as a subgroup of $\PU_2$  were determined by F.\ Klein in his book \cite{klein1956lectures}.
As W.M.\ Winger noted, the sextics among them make up a pencil of genus 10 curves that are  invariant under the  $\mathcal{A}_5$-action. He studied this pencil and  exhibited explicit  equations of all its singular members.

The goal of this paper is to establish the modular properties of a modification of this pencil. This will lead us to
the full classification of genus 10 stable curves with a faithful icosahedral action. We show for example  that  this pencil is universal at every non-singular member, when regarded as a family of genus 10 curves with icosahedral action. The pencil itself is not a moduli space, not even in a coarse sense. First of all there is an unstable member (the triple conic),  which we prove is really like a planar shadow of a (nonplanar) nonsingular genus 10 curve $C_K$ with $\mathcal{A}_5$-action: indeed this conic is triply covered by such a curve. Moreover since the outer automorphism group of $\mathcal{A}_5$ has order 2, precomposing the $\mathcal{A}_5$-action with an automorphism of $\mathcal{A}_5$ that is not inner yields another copy of this family. By replacing the unstable member and taking the disjoint union of two copies, we obtain what looks like a complete family of  $\Acal_5$-curves of genus 10 with some stable degenerations.

Unfortunately this family cannot be universal as the curve $C_K$ admits an $\Acal_5$-automorphism cyclic of order three. We can solve this in two ways. One is to  mark the irregular orbits and require the morphisms respect the marking. By doing so, we can get a fine moduli space and gluing all the local universal deformations we get its universal family. The alternative is that we resort to stacks:  by considering all the smooth genus 10 curves with icosahedral action we get a Deligne-Mumford stack with a 'good' coarse moduli space. The following theorem sum up most of our results (we work over the field $\CC$ throughout).

\begin{thm}\label{main1}
  If $C$ is a smooth projective genus 10 curve endowed with a faithful $\Acal_5$-action, then its orbifold has genus zero and four orbifold points, one  of type
  $\mu_5$ and three of type $\mu_2$. If regard a total ordering of the three orbifold points of type $\mu_2$ as part of the data, then for such objects there exists a fine moduli space whose base $\Mcal_{10;4}^{\Acal_5}$ is a non-singular quasi-projective algebraic variety which has two connected components and comes with an action of $\Scal_3$. It has the following properties
  \begin{enumerate}
    \item[(i)] Orbifold formation with respect to the $\Acal_5$-action defines a degree 20 cover $\Mcal_{10;4}^{\Acal_5}\to  \Mcal_{0,4}$ which is equivariant for the obvious $\Scal_3$-action.
    \item[(ii)] The action of  $\Scal_3$ has exactly one irregular orbit in $\Mcal_{10;4}^{\Acal_5}$. The smooth members of the Winger pencil plus the singular member represented by the triple conic are obtained by passing to the $\Scal_3$-orbit space. This orbit space is also the coarse moduli space $\Mcal_{10}^{\Acal_5}$ of smooth projective genus 10 curves endowed with a faithful $\Acal_5$-action
  \end{enumerate}
  Moreover, $\Mcal_{10}^{\Acal_5}$ admits a Deligne-Mumford compactification $\overline{\Mcal}_{10}^{\Acal_5}$, which parametrizes the stable genus
  10 curves with faithful $\Acal_5$-action that are smoothable as $\Acal_5$-curves. The $\Acal_5$-orbifold formation extends the above morphism  to $\overline{\Mcal}_{10}^{\Acal_5}\to \overline{\Mcal}_0(5; 2,2,2)$, where $\overline{\Mcal}_0(5; 2,2,2)$ is defined by allowing two of the three orbifold points of order $2$ to coalesce. This yields the full Winger pencil.
\end{thm}

This paper is organized as follows. We first review the construction of the Winger pencil. We  give the geometric construction which comes from the irregular orbits of icosahedral group acting on a projective plane and after that we give the construction using invariant theory.
In the second section we prove some cohomology properties of the Winger pencil and show that it is locally universal at its smooth members.
In the third section we show how to replace the triple conic with a smooth algebraic curve, which lies naturally on a non-trivial cone over that conic, so that  projection away from the vertex yields a ramified triple cover and we describe its universal deformation. Finally,  we introduce two moduli functors, one is the moduli functor of non-singular genus 10 curves endowed with a faithful icosahedral group action, another one is the moduli functor not only of the curve but also of a collection of marked irregular orbits. Using  the approach in an article of B.~Farb and E.~Looijenga \textit{Geometry of the Wiman-Edge Pencil and the Wiman Curve} \cite{eduard2019geometry}. We show that the second one is actually a fine moduli and the first one is only a Deligne-Mumford stack but with a `good' coarse moduli, namely the coarse moduli space has two connected components. This leads us to the proof of our main theorem  stated above.

  {We note that there is another classical pencil made of curves with $\Acal_5$-action, namely the Wiman-Edge pencil. This is a pencil of genus 6 curves on the quintic del Pezzo surface whose members inherit from that surface a $\Acal_5$-action. Both the Winger pencil and the Wiman-Edge pencil come from the study of icosahedral symmetric on $\PP^2$ and its invariant sextic curves. However, no relation between them is known. Recently I.\ Dolgachev, B.\ Farb and E.\ Looijenga  in \cite{dolgachev2018geometry} and B.\ Farb and E.\ Looijenga in its sequels \cite{eduard2019geometry} studied the Wiman-Edge pencil in a more modern way. This paper is inspired by their results.}

The author wants to thank the reviewers for their helpful comments and Prof. Eduard Looijenga for his kind guiding and help.

\section{Definition of the Winger pencil}

\subsection{\textbf{Projective Lines and Projective Planes with $\mathcal{A}_5$ Action}}\label{plpp}
By the character theory of finite groups, we have two complex linear representations of $\mathcal{A}_5$ of degree 3 (denoted $I$ and $I'$), whose characters are given by the Table \ref{ctA5}.
\begin{table}
  \centering
  \caption{Characters of two different irreducible $\Acal_5$-representations}\label{ctA5}
  \begin{tabular}{llllll}
    \hline\noalign{\smallskip}
    \       & (1) & (12)(34) & (123) & (12345)                & (12354)                \\
    \noalign{\smallskip}\hline\noalign{\smallskip}
    $I$     & 3   & -1       & 0     & $\frac{1+\sqrt{5}}{2}$ & $\frac{1-\sqrt{5}}{2}$ \\
    $I^{'}$ & 3   & -1       & 0     & $\frac{1-\sqrt{5}}{2}$ & $\frac{1+\sqrt{5}}{2}$ \\
    \noalign{\smallskip}\hline
  \end{tabular}
\end{table}
Here the first line represents all the conjugacy classes of $\mathcal{A}_5$. The representations $I$ and $I'$ are exchanged by an outer automorphism of $\mathcal{A}_5$. Up to conjagacy there is in fact only one such automorphism that is not inner: it is induced by conjugation with an element of $\mathcal{S}_5-\mathcal{A}_5$ and we shall denote any such automorphism by $\iota$. Since we do not always want to make a choice between $I$ and $I'$,  we fix  a 3-dimensional complex vector space $U$ endowed with finite subgroup $\Is\subset \GL (U)$ that acts irreducibly in $U$ and is isomorphic to $\Acal_5$. The choice of an isomorphism $\Is\cong \Acal_5$ then makes of  $U$ an $\Acal_5$-representation isomorphic to $I$ or $I'$. The $\Is$-action on $U$ induces  one on the projective plane  $\PP(U)$.
The group $\Is$ leaves invariant a nondegenerate symmetric bilinear form on $U$. In particular, $U$ is self-dual as a representation of $\Is$.
This form determines a $\Is$-invariant conic $K$ in $\PP(U)$.
Since $K\cong \PP^1$, we can regard  $K$  as a projective representation of $\Is$. In fact, if we endow $U$ with a $\Is$-invariant inner product, then
$\PP(U)$ inherits from this a Fubini-Study metric and the induced metric on $K$ identifies it with a round sphere. On this sphere we can draw
a regular icosahedron such that $\Is$ becomes its group of motions. This makes  following lemma obvious.

\begin{lem}\label{lem:Korbit}
  The group $\Is$ has in  $K$ only 3 irregular orbits. They are of size 12, 20 and 30, corresponding to the vertices, barycentres of faces and midpoints of edges of a spherical icosahedron respectively. $\square$
\end{lem}
\begin{rmk}\label{involution}
  When we  regard $K$ as a round sphere in Euclidean 3-space, then we can define an antipodal map $i$ in $K$. When $K$ is obtained as the projectivization of a complex inner product space of dimension 2, this amounts to assigning to a line its orthogonal complement. So it is anti-linear (hence does not belong to  the image of $\Is$), but still normalizes the $\Is$-action. In fact, it  preserves every irregular orbit of $\Is$ in $K$ and so such an orbit consists of antipodal pairs. See Section 4 of \cite{dolgachev2018geometry}.
\end{rmk}

On the other hand there are no irreducible 2-dimensional complex linear representations of $\Is\cong\mathcal{A}_5$. We can however form a central extension
of $\Is$ by the cyclic group of order 2 that acts on a vector space of dimension 2: if we identify the group of motions of $K$ with $\PSU_2$, then is it given
by a pull-back diagram:
\begin{equation*}
  \xymatrix{1\ar[r]&\mu_2\ar[r]&\SU_2\ar[r]&\PSU_2\ar[r]&1\\
    1\ar[r]&\mu_2\ar[r]\ar[u]&\tilde \Is\ar[u]\ar[r]& \Is\ar[u]\ar[r]&1
  }
\end{equation*}
This gives indeed a non-trivial central extension of $\Is$ by $\mu_2$ and is called the  \textit{binary icosahedral group}. By construction it admits a complex-linear representation of degree 2 (which we shall denote by $V_K$, since it has $K$ as its associated projective line). A point in $P-K$ defines a polar line in $P$ which meets $K$ in a $2$-element subset of $K$ with the same $\Is$-stabilizer. This implies:

\begin{lem}
  The group  $\Is$  has in  $\PP(U)$  6 irregular orbits. The ones in $K$ are  the three orbits of size 12, 20 and 30 mentioned in Lemma \ref{lem:Korbit} and
  the  ones in $P-K$ are obtained as the polar points of the lines spanned by antipodal pairs in an irregular orbit in $K$ and hence are of size
  6, 10 and 15.
\end{lem}

The 6 antipodal pairs that make up the size 12 orbit in $K$ span six distinct lines in $\PP(U)$, no three which are collinear. Their union, which we denote  by $C_\infty$, is clearly a $\Is$-invariant sextic. The \emph{Winger pencil} is a pencil of plane sextics with faithful $\Is$-action generated by the two generators $C_\infty$ and $3K$. For an appropriate  choice of coordinates for $U$, the equation of $K$ is
\begin{equation}\label{defeqK}
  Q(z_0,z_1,z_2)=z_0z_1+z_2^2=0
\end{equation}
while the equation of $C_\infty$ is
\begin{equation}\label{defeqCinfty}
  F(z_0,z_1,z_2)=z_2\Pi_{i=1}^5(\eta^iz_0+\eta^{4i}z_1+z_2)=0
\end{equation}
where $\eta$ is a primitive fifth root of unity. Then this pencil of plane sextics can be described as
\begin{equation}\label{defequ}
  Q^3+\lambda F=0,\ \lambda\in\mathds{P}^1
\end{equation}
Winger \cite{wingerinvariants} showed that this pencil has exactly four singular members. They are
\begin{enumerate}
  \item the conic with multiplicity 3, $3K$,
  \item the union of six lines $C_\infty$,
  \item an irreducible nodal curve with 6 nodes (for $\lambda=-1$) which can be obtained from Bring's curve (of genus $4$) by identifying each  antipodal pair appearing  in one of the two size 12 irregular orbits (we will give more explanation in Remark \ref{norS5act}),
  \item an irreducible nodal curve with 10 nodes (for $\lambda=\frac{27}{5}$) which can be obtained from $K$ by identifying each  antipodal pair appearing in the size 20 orbit.
\end{enumerate}

In the last two cases, the set of nodes make up an  $\Is$-orbit.

\subsection{\textbf{Invariant Theory of Finite Groups}}
This pencil can be described in terms of invariant theory. We recall some theorems that we shall need; for details we refer \cite{neusel2010invariant} and \cite{mukai2003introduction}.
\begin{thm}[Hilbert (1890), Noether (1916, 1926)] Let $G$ be a finite group. Then for every finite dimensional representation $V$ of $G$ then
  \begin{enumerate}
    \item the ring extension $\CC[V]^G\subset \CC[V]$ is finite, and
    \item the invariant ring $\CC[V]^G$ is a finite generated $\CC$-algebra.
  \end{enumerate}
\end{thm}

To determine these generators, we have the following trick. First by the Noether's Normalization Lemma, we can find algebraic independent
homogeneous polynomials $\{f_1,\cdots,f_r\}\subset \CC[V]^G$ such that $\CC[V]^G$ is finite over $\mathds{C}[f_1,\cdots,f_r]$. Such a collection $\{f_1,\cdots,f_r\}$ is called a set of \textit{primary invariants}. According to Hochster-Eagon (Proposition 3 in \cite{hochster1971cohen}), $\CC[V]^G$ is Cohen-Macaulay. This implies that it is a free $\mathds{C}[f_1,\cdots,f_r]$-module of finite rank. A homogeneous basis $\{g_1,\cdots,g_s\}$ of this $\mathds{C}[f_1,\cdots,f_r]$ module is are called a set of \textit{secondary invariants}.

For $R$  a graded noetherian $\mathds{C}$-algebra,  we denote by $H_R(T):=\sum_{d\ge 0}\dim_{\mathds{C}}(R_d) T^d$ its Hilbert Series. The following proposition shows that this can be used to get some information on the degrees of the primary and secondary generators of $\CC[V]^G$.

\begin{prop}
  Let $G$ be a finite group, $V$ a finite dimensional representation of $G$, and  $\{f_1,\cdots,f_r\}$ and $\{g_1,\cdots,g_s\}$ the first and secondary invariants of $G$ in $\CC[V]$ as above. Then
  \begin{equation*}
    H_{\CC[V]^G}(T)=\frac{\sum_{j=1}^{s}T^{\deg(g_j)}}{\Pi_{i=1}^{r}(1-T^{\deg(f_i)})}.
  \end{equation*}
\end{prop}

On the other hand, we can compute the  Hilbert Series by means of
\begin{thm}[Molien's Formula] We have
  \begin{equation*}
    H_{\CC[V]^G}(T)=\frac{1}{|G|}\sum_{\pi\in G}\frac{1}{\det_{V}(1-T\pi)}=\frac{1}{|G|}\sum_{\pi\in G}\frac{1}{\Pi_i(1-\lambda_i^{\pi}T)},
  \end{equation*}
  where the $\lambda_i^{\pi}$'s  are the eigenvalues of $\pi$
\end{thm}
So we can determine the degree of invariant polynomials by computing the Hilbert series in two different ways. Using this we can calculate that $\CC[U]^{\Is}$ has primary invariants of degree 2 (denoted by $\alpha$), 6 (denoted by $\beta$) and 10 and secondary invariants of degree 0 and 15.

\begin{cor}\label{cor:Wingergen}
  The Winger pencil is defined by $\CC[U]_6^{\Is}$ and admits $\alpha^3,\beta$ as a basis.
\end{cor}

We close this section with mentioning that we have a 6 dimensional irreducible representation $E$ of $\mathcal{S}_5$ with the character table (Table \ref{chartabS5}).
\begin{table}
  \center
  \caption{Characters of six dimensional representation of $\Scal_5$}\label{chartabS5}
  \begin{tabular}{llllllll}
    \hline\noalign{\smallskip}
        & (1) & (12) & (12)(34) & (123) & (123)(45) & (1234) & (12345) \\
    \noalign{\smallskip}\hline\noalign{\smallskip}
    $E$ & 6   & 0    & -2       & 0     & 0         & 0      & 1       \\
    \noalign{\smallskip}\hline
  \end{tabular}
\end{table}
If we consider it as a representation of $\mathcal{A}_5$, then it is the direct sum of two distinct 3 dimensional irreducible representations of $\mathcal{A}_5$, i.e $E\cong I\oplus I'$.

\section{Versality of the Winger Pencil at its smooth members}

Let us define the family $\mathscr{W}\rightarrow \mathscr{B}$ as the union of two copies of Winger pencil, obtained by two identifications
$\Acal_5\cong \Is$ that differ by an automorphism $\iota$ of $\Acal_5$ that is not inner. So every member of this  family except for the two copies of $3K$ is a curve of arithmetic genus 10 and endowed with a faithful $\mathcal{A}_5$-action and each component of $\mathscr{W}\rightarrow \mathscr{B}$ has 4 singular fibers. We will use $\mathscr{W}^{\circ}\rightarrow B^{\circ}$ to denote the locus  over which  this morphism is  smooth.

\begin{lem}\label{h1CC}
  Let $C$ be a  smooth member of the  family $\mathscr{W}$. Then $H^0(C,\omega_{C})$ is a 10-dimensional complex representation of $\mathcal{A}_5$ that is isomorphic to $V\oplus I\oplus I'$ where $V$ is a $4$-dimensional irreducible representation. Moreover $H^1(C;\CC)$ is isomorphic to $V^{\oplus2}\oplus (I\oplus I')^{\oplus 2}$.
\end{lem}
\begin{proof}
  Choose a generator  $\mu$ of $\wedge^3 U^\vee$ and regard this as a translation-invariant $3$-form on $U$.
  Note that this form is $\Is$-invariant.
  Let $F\in \CC[U]_6^\Is$ be a defining equation for $C$. Then a  meromorphic 3-form on $U$ of the type
  \begin{equation*}
    \Omega_G=G\frac{\mu}{F}, \quad \text{ with } G\in  \CC[U]_3
  \end{equation*}
  has a simple pole along the cone over $C$ (defined by $F=0$) and at the hyperplane at infinity  $\PP (U)$.
  This enabled us to define a linear  map
  \begin{equation*}
    \begin{aligned}
      \CC[U]_3 & \rightarrow H^0(C,\omega_{C})        \\
      G        & \mapsto \Res_C\Res_{\infty}\Omega_G.
    \end{aligned}
  \end{equation*}
  This map is known to be injective see Chapter 5 of \cite{griffiths2014principles} and since
  both $\CC[U]_3$ and $H^0(C,\omega_{C})$  are of dimension $10$,  it  must be an isomorphism.
  As it is also $\Is$-equivariant, it follows  that  $H^0(C,\omega_{C})$ is as a $\Is$-representation isomorphic  with $\CC[U]_3$.
  We know that $U^{\vee}$ is either $I$ or $I'$ as an $\mathcal{A}_5$-representation. Recall the following character formula for symmetric powers
  \begin{equation*}
    \chi_{\CC[U]_3}(g)=\frac{\chi_{U^{\vee}}^3(g)+3\chi_{U^{\vee}}(g^2)\chi_{U^{\vee}}(g)+2\chi_{U^{\vee}}(g^3)}{6}
  \end{equation*}
  For $\Acal_5$, we can compute the characters see Table \ref{ctsc}.
  \begin{table}
    \centering
    \caption{Characters of Symmetric Cubics}\label{ctsc}
    \begin{tabular}{llllll}
      \hline\noalign{\smallskip}
      \          & (1) & (12)(34) & (123) & (12345) & (12354) \\
      \noalign{\smallskip}\hline\noalign{\smallskip}
      $\Sym^3I$  & 10  & -2       & 1     & 0       & 0       \\
      $\Sym^3I'$ & 10  & -2       & 1     & 0       & 0       \\
      \noalign{\smallskip}\hline
    \end{tabular}
  \end{table}
  Note that if we denote $g:=(12345)$ and $h:=(12354)$, then $g^2$ and $g^3$ are both in the conjugacy class of $h$. This is because first we have $g^2=(13524)$ and equation
  \begin{equation*}
    (43)(15)(13524)(15)(43)=(12354)
  \end{equation*}
  So if we identify $\Is$ with $\Acal_5$, then the character table  of $\CC[U]_3$ as a $\Acal_5$-representation
  does not depend on the choice of an isomorphism $\Is\cong\Acal_5$. We  find that $\CC[U]_3$ is isomorphic to $V\oplus I\oplus I'$ as a $\mathcal{A}_5$ representation, where $V$ is the  irreducible representation of $\mathcal{A}_5$ of degree $4$. From the Hodge decomposition, we have $H^0(C; \CC)\cong V^{\oplus2}\oplus (I\oplus I')^{\oplus 2}$
\end{proof}
{We have a more explicit description on the homology of the smooth members of the Winger pencil.}
\begin{lem}\label{lemma:inthomology}
  There exists a $\ZZ\Is$-module $L$ which is free as a $\ZZ$-module of rank $10$ and admits a  $\ZZ$-basis
  $\alpha_1, \dots , \alpha_{10}$ such that the  20-element set $\{ \pm \alpha_i\}_{i=1}^{10}$ is an $\Is$-orbit. This $\ZZ\Is$-module is unique up to isomorphism.  If $C$ is a smooth member of the Winger pencil, then there exists an  $\Is$-equivariant embedding  $L\hookrightarrow H_1(C)$  whose image
  is a  Lagrangian sublattice and for which the cokernel can be $\Is$-equivariantly identified with $L^\vee$.
\end{lem}
\begin{proof}
  The map $\mathscr{W}\rightarrow B$ is proper, and hence topologically locally trivial over the locus where it is smooth. So it is enough to prove this for just one of its smooth members. So without loss of generality we can assume that $C=C_t$ with $t$  close to $\frac{27}{5}$. Recall that the singular set $N$  of $C_{\frac{27}{5}}$ consists of 10 nodes and that the normalization of $C_{\frac{27}{5}}$ is a projective line. The nodes give 10 `vanishing circles' whose union $A$ has the property that $H_*(C; A)$ is naturally (hence $\Acal_5$-equivariantly) identified  with $H_*(C_{\frac{27}{5}}; N)$. The long exact homology sequence of the pair $(C, A)$
  \begin{equation*}
    \xymatrix{0\ar[r]&H_2(C)\ar[r]&H_2(C;A)\ar[r]&H_1(A)\ar[r]&H_1(C)\ar[r]&H_1(C;A)\ar[r]&\cdots}
  \end{equation*}
  shows that the map $H_1(A)\to H_1(C)$ is injective and that its  cokernel embeds in  $H_1(C_{\frac{27}{5}}; N)$. The latter is torsion free and so the  image $L$ of  $H_1(A)\to H_1(C)$ is a primitive sublattice of $H_1(C)$ of rank 10. It is also clear that the intersection product is identically zero on $L$, so that $L$ is in fact a Lagrangian submodule of
  $H_1(C)$. This gives the identification of $H_1(C)/L\cong L^{\vee}$. We thus obtain an exact sequence
  \begin{equation*}\label{111}
    \xymatrix{0\ar[r]&L\ar[r]&H_1(C)\ar[r]&L^{\vee}\ar[r]&0}.
  \end{equation*}
  Now $L$ contains 20 vanishing cycles as 10 antipodal pairs. Since  $\Is$ acts  transitively on $N$,  the vanishing cycles make up either a single $\Is$-orbit of 20 elements or two 10-element orbits $\Is$-orbit that are opposite each other. In the last case, the sum of the elements of one such orbit would give a nonzero $\Is$-invariant  element of $H_1(C)$ and this would contradict Lemma \ref{h1CC}. Hence the 20  vanishing cycles make up one orbit. The uniqueness comes from the fact that $\mathcal{A}_5$ has an unique subgroup who has 20 different conjugacy classes.
\end{proof}

Each smooth member $C$ of our pencil is $\Is$-equivariantly embedded in $\PP(U)$. This embedding is defined by a complete linear system
that is $\Is$-invariant. The following proposition states that such linear system is unique.

\begin{prop}\label{unilib}
  A smooth member $C$ of the family $\mathscr{W}$ admits {exactly} one very ample complete linear system of dimension $2$.
\end{prop}
\begin{proof}
  This is Lemma 2.1 in \cite{tyurin1975intersection}.
\end{proof}

We derive from this uniqueness property two corollaries.

\begin{cor}\label{nonextension}
  For no member of the family $\mathscr{W}$, the  $\mathcal{A}_5$-action extends to an $\mathcal{S}_{5}$-action.
\end{cor}

The proof uses the following lemma.

\begin{lem}\label{s5p1p2}
  A nontrivial (hence faithful) action   of $\Acal_5$ on $\mathds{P}^1$  resp.\  $\mathds{P}^2$ does not extend to $\mathcal{S}_5$.
\end{lem}
\begin{proof}
  The classification of all projective representations of symmetric and alternating groups is done by Schur in \cite{schur1911darstellung}. We suggest an English version of this theory by \cite{hoffman1992projective}.  By the Example of \cite{hoffman1992projective} in Page 79, all the projective representations of $\mathcal{S}_5$ are of degree 1, 4, 5 and 6. So there is no faithful $\mathcal{S}_5$ action of degree 2 or 3.
\end{proof}

\begin{proof}[Proof of Corollary \ref{nonextension}]
  If $3K$ admits a faithful $\mathcal{S}_5$-action, then $\mathcal{S}_5$ will act faithfully on a regular icosahedron. This gives an injective group homomorphism from $\mathcal{S}_5$ to the orientation preserving automorphisms of this icosahedron. But we know the latter group is $\Acal_5$ and so  this is impossible.

  For $C_{\infty}$ we prove this as follows. Recall that $C_{\infty}$ consists of $6$ lines. Since  $\Acal_5$ acts transitively on these, the stabilizer
  of one such  line $\ell$  must be order $10$. The subgroups of  $\Acal_5$ of order $10$ are dihedral and make up a single conjugacy class.
  The five other lines meet $\ell$ in as many distinct points and for a suitable affine coordinate on $\ell$ they are given by the $5th$ roots of unity.
  Since the M\"obius transformations of $\PP^1$ preserving the $5th$ roots of unity form a dihedral group of order $10$,  $\Acal_5$ must be the full
  automorphism group of $C_{\infty}$. In particular, the $\Acal_5$-action on $C_{\infty}$ cannot extend to an $\mathcal{S}_5$-action.

  Suppose that a singular member with 6 nodes,  $C_{-1}$, say, admits a $\mathcal{S}_5$-action. Then the 6 nodes must form an irregular
  $\mathcal{S}_5$-orbit. We normalize $C_{-1}$ to get a nonsingular algebraic curve of genus 4. The $\mathcal{S}_5$-action lifts naturally by the universal property of normalization. Then the 6 nodes become a size 12 orbit of $\mathcal{S}_5$. So the stabilizer of a point of this orbit has order $10$.
  Since the point stabilizers of   a faithful finite group action on a smooth curve are cyclic (see  Lemma \ref{faithstab} below), this would imply that
  $\mathcal{S}_5$ contains an element  of order 10, which is clearly not the case. So the $\Acal_5$-action on $C_{-1}$ does not extend to a $\mathcal{S}_5$-action.

  Suppose that a singular member with 10 nodes admits a $\mathcal{S}_5$-action. After normalizing it, we get a non-singular algebraic curve of genus zero,  i.e.,  a copy of  $\mathds{P}^1$, to which the  $\mathcal{S}_5$-action lifts naturally. By Lemma \ref{s5p1p2} this is impossible.

  Finally, suppose  that for a smooth member $C$ of the Winger  family, the $\mathcal{A}_5$-action extends to a $\mathcal{S}_5$-action.
  Since $\Acal_5$ is normal in $\mathcal{S}_5$, any  element of $\mathcal{S}_5$ then takes the linear system of linear sections to an $\Acal_5$-invariant complete linear system of dimension 2.
  By Proposition \ref{unilib} this linear system is the same. This proves that this linear system is in fact $\mathcal{S}_5$-invariant. So the $\mathcal{S}_5$-action on $C$ then extends to an action on $\PP(U)$. But this is impossible by Lemma \ref{s5p1p2}.
\end{proof}

\begin{rmk}\label{norS5act}
  Although $C_{-1}$ admits no $\Scal_5$-action, its normalization $\tilde{C}_{-1}$ admits a $\Scal_5$-action. Consider the projective space $\PP W_4$ with faithful $\Scal_5$-action where $W_4$ is the dimension 4 irreducible representation of $\Scal_5$. The curve $\tilde{C}_{-1}$ can be obtained as the complete intersection of a $\Scal_5$-invariant quadratic surface and a $\Scal_5$-invariant cubic surface. The smooth curve $\tilde{C}_{-1}$ has an irregular orbit of size 24, while treating as $\Acal_5$-curve this irregular orbit splits into two size 12 orbits. Hence such $\Scal_5$-action doesn't descends to $C_{-1}$. Such curve is also called Bring's Curve. For more details see Chapter 5 of \cite{Cheltsov2015Cremona}.
\end{rmk}

\begin{cor}\label{a5equiv}
  Any two distinct members of the family $\mathscr{W}$ are not isomorphic as genus 10 curves endowed with a faithful $\mathcal{A}_5$-action, i.e., we cannot find an $\mathcal{A}_5$-equivariant isomorphism between them.
\end{cor}
We begin with a lemma {that is a special case of what is shown Section 2.2 of \cite{akhiezer2012lie} and} which will also be used later

\begin{lem}\label{faithstab}
  Let $G$ be a finite group acting faithfully on a connected nonsingular complex curve $C$. Then the stabilizer $G_x\ (x\in C)$ is cyclic and comes with a canonical generator $\mu$ such that for any $G$-automorphism $h$ of $C$ conjugation with $h$ takes the canonical generator of $G_x$ to the one of $G_{h(x)}$.
\end{lem}

\begin{proof}[Proof of Corollary \ref{a5equiv}]
  Suppose first that $C$ and $C'$  are smooth members and that there exist a  $\Acal_5$-equivariant  isomorphism $\psi:C\rightarrow C'$. We prove that they are then equal.
  Consider the two $\Acal_5$-invariant linear systems on $C$ that define the embedding of $C$ in $\PP(U)$ and the composite of $\psi$ with the
  embedding of $C'$ in $\PP(U')$ ($U'$ is also a 3-dimensional complex vector space endowed with finite subgroup $\Is\subset \GL (U')$ that acts irreducibly in $U'$ and is isomorphic to $\Acal_5$). Proposition \ref{unilib} tells us that the two linear systems coincide. In other words, $\psi$ is induced
  by a projective-linear transformation $\PP(U)\rightarrow\PP(U')$. This  transformation is  $\mathcal{A}_5$-invariant, but since $U$ and $U'$ are irreducible as
  $\mathcal{A}_5$-representations, Schur's Lemma then implies that $U\cong U'$ and this transformation must be the identity. So $C=C'$ as algebraic curves with $\mathcal{A}_5$-action.

  If we take two singular members which are not isomorphic as algebraic curves, then clearly they are not isomorphic as algebraic curves with $\mathcal{A}_5$-action. Hence we only need to concentrate on different singular members which are isomorphic as algebraic curves.

  Let $C$ be a singular member of Winger pencil, not equal to $3K$. Two different $\Acal_5$-actions on $C$ give two different group embeddings $\rho:\Acal_5\to \Aut(C)$ and $\rho':\Acal_5\to \Aut(C)$. If the claim doesn't hold, then we can find $\psi\in \Aut(C)$ such that
  \begin{equation}\label{eqrho1}
    \rho'_g=\psi^{-1}\circ\rho_g\circ\psi
  \end{equation}
  for all $g\in\Acal_5$.

  On the other hand, by the definition, $\rho$ (resp. $\rho'$) factors through $\tilde \rho:\Acal_5\to \Aut^{C}(\PP^2)$ (resp. $\tilde{\rho}'$) and restriction to $C$, where $\Aut^{C}(\PP^2)$ consists of the automorphisms of $\PP^2$ which leave $C$ invariant. The morphisms $\tilde{\rho}$ and $\tilde{\rho}'$ satisfy the condition that $\tilde{\rho}'_g=\tilde{\rho}_{\iota^{-1}g\iota}$ for some $\iota\in\Scal_5-\Acal_5$ and all $g\in\Acal_5$. This implies that for some $\iota\in\Scal_5-\Acal_5$ and all $g\in\Acal_5$,the morphisms $\rho$ and $\rho'$ must satisfy the equation
  \begin{equation}\label{eqrho2}
    \rho'_g=\rho_{\iota^{-1}g\iota}
  \end{equation}

  Let $\{p_1,\cdots,p_{12}\}$ be the size 12 irregular orbit. The stabilizer of each $p_i$ is isomorphic to cyclic group of order five $\mu_5$. By the Lemma \ref{faithstab} this cyclic group admits a canonical generator we denote it by $g_{p_i}$. The Equation (\ref{eqrho1}) implies that $g_{\psi(p_i)}=\psi\circ g_{p_i}\circ\psi^{-1}$ equals the canonical generators $g'_{p_i}$ with respect to $\rho'$. But Equation (\ref{eqrho2}) implies that $g_{\psi(p_i)}$ and $g'_{\psi(p_i)}$ are conjugate by an element $\iota\in\Scal_5-\Acal_5$, so that they lie in different $\Acal_5$-conjugacy classes, in particular they cannot be identified through an automorphism of $C$. This is a contradiction.

  When $C$ isomorphic to $3K$, the $\Acal_5$-action on $3K$ is actually an $\Acal_5$-action on $K$. Recall that $K=\PP V_K$ embeds into $\mathds{P}^2=\PP(\Sym^2V_K)$ as the image of Veronese map where $V_K$ is a 2 dimensional irreducible $\tilde\Acal_5$-representation. Hence $\psi$ lifts naturally to a linear transformation between two different $\Acal_5$-representations. Then Schur's Lemma implies that this $\psi$ doesn't exist. This finishes the proof.
\end{proof}

\begin{thm}\label{universal}
  The family $\mathscr{W}^{\circ}\rightarrow \mathscr{B}^{\circ}$ is locally universal: it induces an universal deformation of each member, when regarded as an $\mathcal{A}_5$-curve.
\end{thm}
\begin{proof}
  Let $C$ be a smooth member of the Winger pencil. Denote by
  $\nu_{C}$ the normal bundle of $C$ in $\PP(U)$, so that we have the exact sequence
  \[
    0\to \theta_C\to \Ocal_C\otimes \theta_{\PP(U)}\to \nu_C\to 0
  \]
  of $\Ocal_C$-modules. The proof will center around the associated long exact sequence
  \[
    0\to H^0(C,\theta_C)\to H^0(C,\Ocal_C\otimes \theta_{\PP(U)})\to H^0(C,\nu_C)\to H^1(C,\theta_C)\to\cdots ,
  \]
  or rather its $\Is$-invariant part.
  It fact, we have to interpret some of its terms geometrically.  As is well-known, $H^0(C,\nu_C)$ is the space of first order deformations of $C$ in $\PP(U)$, whereas  $H^1(C,\theta_C)$ is the  space of first order deformations of $C$ as a curve and the coboundary map $H^0(C,\nu_C)\to H^1(C,\theta_C)$ is the obvious map. Hence $\Is$-invariant part of $H^1(C,\theta_C)$ is
  the  space of first order deformations of $C$ which retain the $\Is$-action. We claim that the $\Is$-invariant part of $H^0(C,\nu_C)$ is the tangent space at $[C]$ of the family $\mathscr{W}$.

  To show this, note that we have a natural map
  \begin{equation}\label{tgh0}
    T_{[L]}\PP(\CC[U]_6)=\Hom_\CC(L, \CC[U]_6/L)\to H^0(C,\nu_C)
  \end{equation}
  where $L$ is the line of equations for $C$. By the adjunction formula, we see that $\nu_C$ has degree $36$ and then  Riemann-Roch implies that $H^0(C,\nu_C)$ is of dimension $1-10+36=27$. On the other hand  $\CC[U]_6$ is of dimension $\binom{8}{2}=28$. The map (\ref{tgh0}) is injective, hence it must be a $\Is$-equivariant isomorphism. Now $T_{[C]}\mathscr{W}$ is the $\Is$-invariant part of $T_{[L]}\PP(\CC[U]_6)$ thus isomorphic to $H^0(C,\nu_C)^{\Is}$.

  So what we must prove is that the map of invariants $H^0(C,\nu_C)^\Is\to H^1(C,\theta_C)^\Is$ is an isomorphism. We do this by first showing that this map is injective  and then proving that $H^1(C,\theta_C)^\Is$ is of dimension one (which we relegate to  Lemma \ref{lemma:dim1} below).

  To prove injectivity, {we first observe that $H^0(C,\theta_C)=0$ (for $C$ has negative Euler characteristic), hence it suffices to show that $H^0(C,\Ocal_C\otimes \theta_{\PP(U)})^\Is=0$.}
  Consider  the standard exact sequence of $\Ocal_{\PP(U)}$-modules
  \[
    0\to \Ocal_{\PP(U)}\to \Ocal_{\PP(U)}(1)\otimes_\CC U\to  \theta_{\PP(U)}\to 0.
  \]
  This remains exact after tensoring with $\Ocal_C$. The associated long exact sequence begins as
  \[
    0\to \CC\to H^0(C,\Ocal_{C}(1))\otimes_\CC U\to  H^0(C,  \Ocal_C\otimes\theta_{\PP(U)})\to  H^1(C, \Ocal_C)\to \cdots
  \]
  Its $\Is$-invariant part is still exact. We have  $H^0(C,\Ocal_{C}(1))\otimes_\CC U\cong U^\vee\otimes_\CC U=\End(U)$. Since $U$ is
  irreducible  as an $\Is$-representation, Schur's lemma implies that $\End(U)^\Is$ consists of the scalars, in other words, is the image of
  $\CC\to H^0(C,\Ocal_{C}(1))\otimes_\CC U\cong \End(U)$. We established in Lemma \ref{h1CC} that $H^1(C, \Ocal_C)^\Is=0$, and so it follows  that
  $H^0(C,\Ocal_C\otimes \theta_{\PP(U)})^\Is=0$ as desired.

  Lemma \ref{lemma:dim1} below will then finish the proof.
\end{proof}

\begin{lem}\label{lemma:dim1}
  Let $C\subset \PP(U)$ be a smooth member of the family $\mathscr{W}$. Then its space of first order deformations as an $\Is$-curve is
  of dimension one.
\end{lem}
\begin{proof}
  Recall that the first order deformations of $C$ are parametrized by $H^1(C,\theta_C)$, where $\theta_C$ is the sheaf of tangent fields on $C$. Those that are  $\Is$-invariant are parametrized  by the $\Is$-fixed part of $H^1(C,\theta_{C})$. Serre duality identifies the dual of $H^1(C,\theta_C)$ with
  $H^0(C,\omega_{C}^{\otimes 2})$, even as a $\Is$-representation. It therefore suffices to show that the $\Is$-invariant part of $H^0(C,\omega_{C}^{\otimes 2})$ is of dimension one.
  Since $C$ is smooth, the  residue  sequence for the closed immersion $C\subset  \PP(U)$,
  \[
    \xymatrix{0\ar[r]& \omega_{\PP(U)}\ar[r]&\omega_{\PP(U)}(C)\ar[r]&\omega_{C}\ar[r]&0},
  \]
  is an exact sequence of  $\Ocal_{\PP(U)}$-modules. The residue homomorphism $\omega_{\PP(U)}(C)\to \omega_{C}$ induces
  a  homomorphism $\omega_{\PP(U)}^{\otimes 2}(2C)\to \omega^{\otimes 2}_{C}$ which then fits in the exact sequence of $\Ocal_{\PP(U)}$-modules
  \begin{equation}
    \xymatrix{0\ar[r]& \omega_{\PP(U)}^{\otimes2}(C)\ar[r]&\omega^{\otimes2}_{\PP(U)}(2C)\ar[r]&\omega_{C}^{\otimes 2}\ar[r]&0}.
  \end{equation}
  This gives the long exact sequence:
  \begin{multline*}
    0\to H^0(\PP(U),\omega_{\PP(U)}^{\otimes2}(C))\to H^0(\PP(U),\omega^{\otimes2}_{\PP(U)}(2C))\to \\
    \to H^0(C,\omega_{C}^{\otimes 2})\to H^1(\PP(U),\omega_{\PP(U)}^{\otimes2}(C))\to\cdots
  \end{multline*}
  Note that
  \begin{equation*}
    \omega_{\PP(U)}^{\otimes2}(C)\cong \Ocal_{\PP(U)}(-3)^{\otimes2}\otimes\Ocal_{\PP(U)}(6)\cong \Ocal_{\PP(U)}
  \end{equation*}
  and that
  \begin{equation*}
    \omega^{\otimes2}_{\PP(U)}(2C)\cong \Ocal_{\PP(U)}(6).
  \end{equation*}
  It follows that  $H^0(\PP(U),\omega_{\PP(U)}^{\otimes2}(C))$ is of dimension one (and hence trivial as a $\Is$-representation), $H^1(\PP(U),\omega^{\otimes 2}_{\PP(U)}(C))=0$ and
  $H^0(\PP(U),\omega^{\otimes2}_{\PP(U)}(2C))\cong\CC[U]_6$ as a $\Is$-representation. So it remains to see that $\CC[U]_6^\Is$ is of dimension 2.
  But this  we noted in Corollary \ref{cor:Wingergen}. \end{proof}

\section{Modification of Winger Pencil}
As we observed before, every member in the Winger Pencil is a stable curve except $3K$. In this section, we modify the Winger pencil with $3K$ replaced by a nonsingular algebraic curve which is a ramified Galois cover of $K$ with Galois group the group of third roots of unity $\FFF$.

\begin{thm}\label{tricov}
  There exists a ramified $\FFF$-cover of $K$  whose total space $C_K$ is a non-singular irreducible curve of genus 10 to which the $\Is$-action on $K$ lifts in a unique way. Such a  covering is unique up to isomorphism and ramifies at the size 12 irregular $\Is$-orbit and nowhere else.
\end{thm}

To prove this theorem, let us begin with some lemmas about group cohomology. In what follows we regard $\FFF$ as a trivial $\Is$-module (so that it is also trivial as a $\tilde\Is$-module).

\begin{lem}\label{cohA5Z}
  Both  $H^1(\tilde\Is,\FFF)$ and $H^2(\tilde\Is,\FFF)$ are trivial.
\end{lem}
\begin{proof}
  Consider the quaternions  $\HH$. The unit quaternions $\HH_1$ form a Lie group isomorphic to $\SU_2$ with a $3$-sphere as underlying manifold.
  So the universal principal $\SU_2$-bundle $E\SU_2\to B\SU_2$ is realized by the unit $3$-sphere bundle in the tautological bundle over $\PP(\HH^\infty)=\cup_n \PP(\HH^n)$.
  If we pass to the orbit space of  $\tilde\Is\hookrightarrow \SU_2$, we obtain a fibration $B\tilde\Is\to B\SU_2$ with fiber  $\SU_2/\tilde\Is$.
  It is clear that $\tilde{\Is}$ acts on $S^3\cong \SU_2$ freely, hence $S^3\to \SU_2/\tilde\Is$ is an universal cover. Its fundamental group $\pi_1(\SU_2/\tilde\Is)$ is isomorphic to $\tilde{\Is}$.
  The first homology group $H_1(\SU_2/\tilde\Is,\ZZ)$ is the abelianization of $\tilde\Is$ and since $ \tilde\Is$  is perfect, this is zero. The second homology group $H_2(\SU_2/\tilde\Is),\ZZ)=H^1(\SU_2/\tilde\Is,\ZZ)$ by Poincar\'e duality and the last group is certainly zero (because $\tilde\Is$ is finite). Therefore by the universal coefficient theorem, the $\FFF$-coefficient cohomology of $\SU_2/\tilde\Is$ is as following:
  \[
    H^q(\SU_2/\tilde\Is,\FFF)\cong
    \begin{cases}
      \FFF, & \text{when $q=0, 3$ and} \\
      0     & \text{otherwise.}
    \end{cases}
  \]

  The cohomology ring $H^*( B\SU_2; \mathds{Z})\cong H^*(\PP(\HH^\infty);\mathds{Z})$ is  known to be a polynomial ring $ \mathds{Z}[v]$ with its generator in  degree 4 (See Page 222 of \cite{hatcher2002algebraic}).

  So the  Serre spectral sequence for fibrations
  \begin{equation*}
    E_2^{p,q}=H^p(B\SU_2,H^q(\SU_2/\tilde\Is,\FFF))\Rightarrow H^{p+q}(B\tilde\Is,\FFF)=H^{p+q}(\tilde\Is,\FFF),
  \end{equation*}
  {has the property that
  \[
    E_2^{p,q}\cong
    \begin{cases}
      \FFF, & \text{when $p=0$ or $p=3$ and $q=4k$ where $k\in \ZZ_{\ge 0}$ and} \\
      0     & \text{otherwise.}
    \end{cases}
  \]}
  We see that  $E_{2}^{p, q}$ is trivial when $0<p+q<3$. It follows that $H^{1}(\tilde\Is,\FFF)=H^{2}(\tilde\Is,\FFF)=0$.
\end{proof}

\begin{lem}\label{lemma:isplit}
  Both  $H^1(\Is,\FFF)$ and $H^2(\Is,\FFF)$ are trivial and (hence)  any group extension of $\Is$ by $\FFF$ splits
  and the splitting is unique.
\end{lem}

\begin{proof}
  Since $\tilde\Is$ is an extension  of $\Is$ by $\mu_2$, we have
  the  Lyndon-Hochschild-Serre spectral sequence:
  \[
    E^{p,q}_2=H^p(\Is, H^q(\mu_2, \FFF))\Rightarrow H^{p+q}(\tilde\Is, \FFF).
  \]
  The  classifying space of $\mu_2$ is the infinite  real projective space whose integral cohomology is $\ZZ$ in degree zero and cyclic of order $2$ in positive  degrees. It follows that  $H^q(\mu_2, \FFF)$ is nonzero only when $q=0$; it is then a copy of $\FFF$. So the above sequence degenerates on its second page, so that we have an isomorphism $H^p(\Is, \FFF))\cong H^p(\tilde\Is, \FFF)$. Now apply Lemma \ref{cohA5Z}.
\end{proof}

We can now  prove  Theorem \ref{tricov}.

\begin{proof}[Proof of Theorem \ref{tricov}]
  We first prove  the existence of such a cover. We know that the faithful $\Is$-action on $K$ has an irregular orbit $Q$ of size 12.
  If $q_0$ is a base point of $K-Q$, then
  a connected $\FFF$-cover of $K-Q$ is given by a surjective group homomorphism $\phi:\pi_1(K-Q,q_0)\rightarrow \FFF$.  Such a group
  homomorphism factors through $H_1(K-Q)$ and hence is also given by a nonzero  homomorphism $\phi: H_1(K-Q)\to  \FFF$.
  The group $H_1(K-Q)$ is generated by the  simple positive loops around the points of $Q$. These 12 generators are subject to the relation
  that their sum is zero and this defines a presentation of $H_1(K-Q)$. We  then take as our cover the one defined by assigning to every generator
  the  element $1\in\FFF$ (since $12=0$ in $\FFF$, this is well-defined). By the Riemann extension theorem
  this extends to a nonsingular $\FFF$-cover  $p:C_K\rightarrow K$. The construction is canonical in the sense that
  it is only in terms of the pair $(K, Q)$: any automorphism of $K$ which preserves $Q$ will preserve the element of
  $\Hom (H_1(K-Q); \FFF)$ defined above and will therefore lift over $p$ to an automorphism of $C_K$.
  This applies in particular to elements of $\Is$. Denote by $\Aut_p(C_K)$ the group of automorphisms $f$ of $C_K$ such that there exist a  $\bar f\in \Is$ which makes the following diagram commute:
  \begin{equation*}
    \xymatrix{C_K\ar[r]^{f}\ar[d]^p&C_K\ar[d]^p\\
      K\ar[r]^{\bar f}&K
    }
  \end{equation*}
  It is easy to see that for a given $f$ the associated $\bar f$ must be unique. From what we showed above, it follows that we have a surjective group homomorphism $p_{\ast}:\Aut_p(C_K)\rightarrow \Is$. Its kernel consists of the covering transformations of $p$, so  that we have an exact sequence
  \begin{equation*}
    \xymatrix{0\ar[r]&\FFF\ar[r]&\Aut_p(C_K)  \ar[r]^{\quad\quad p_*} &\Is \ar[r]&1}.
  \end{equation*}
  This is a central  extension of $\Is$ by $\FFF$. By Lemma \ref{lemma:isplit}, such an extension is  split in a unique manner
  so that we have in fact a unique way of letting $\Is$ act on $C_K$ such that $p$ is equivariant.

  We next show that $C_K\to K$ is unique up to equivariant isomorphism.
  Suppose we have another smooth connected $\FFF$-cover $p':C'_K\rightarrow K$ with $\Is$-action. Every ramification  point  of $p'$ will be a point
  of total ramification, for $\FFF$ has no proper nontrivial subgroups. The set of ramification points is $\Is$-invariant, so is a  union of $\Is$-orbits.
  Since $C_K$ is smooth and connected,  it is irreducible. The Riemann Hurwitz formula
  \begin{equation*}
    \sum_{P\in C_K}(e_P-1)=24
  \end{equation*}
  then shows that the only possibility is that we have 12 points of total  ramification (the ramification index being 3). Since all the $\Is $-orbits in $K$ distinct from $Q$ have size
  $>12$, the branch locus must be $Q$.  The restriction  of $p'$ to $K-Q$ defines and is  defined by a homomorphism $\phi': H_1(K-Q)\to \FFF$. Since the covering is connected this homomorphism is onto and since the covering comes with an $\Is$-action, it is also $\Is$-invariant. In particular, will take on the $12$ generators the same value and this value must be nonzero. This just means that $\phi'=\pm \phi$. The two coverings are then isomorphic, as they are
  identified  by the identity of $\FFF$ or  the inversion automorphism $a\in\FFF\mapsto -a\in\FFF$.
  This proves the uniqueness.
\end{proof}

Recall that the Winger  pencil defines a hypersurface $\Cs\subset \PP(U)\times \PP^1$ by the Equation (\ref{defequ}) with the fiber $\Cs_o$ over $o=[0:1]\in\PP^1$ is $3K$. It gives a deformation of $3K$ in $\PP(U)$. Next we will show that this family is closely related to the local deformation property of $C_K$, namely by taking proper base change and blow up along unexpected members in $\mathscr{C}$ we get the universal deformation of $C_K$.

Let us begin with describing a basic example of a birational transformation over a disk. Let $\Delta$ be the complex unit disk (with origin $o$),  $\Ps\to \Delta$ a bundle of projective spaces, and $H$ a hyperplane in the fiber $P_o$. The normal bundle of $H$ in  $\Ps$ is $\Ocal_{H}(1)\oplus \Ocal_{H}$ (we get this decomposition after we have chosen a trivialization of $\Ps/\Delta$) and so its blowup
$\Bl_{H}(\Ps)\to \Ps$ has exceptional divisor $\PP(\Ocal_{H}(1)\oplus \Ocal_{H})$.
The summand  $\Ocal_{H}(1)$  resp.\  $\Ocal_{H}$ determines a section
$\sigma_\infty$ resp.\ $\sigma_0$ of  this $\PP^1$ bundle with normal bundle $\Ocal_H(-1)$  resp.\ $\Ocal_H(1)$.
We can contract $\sigma_\infty$ in this exceptional divisor to get $\PP(\Ocal_{H}(1)\oplus \Ocal_{H})\to \hat P_o$,  with
$\hat P_o$ a copy of a projective space, which  has the image of $\sigma_0(H)$ as a hyperplane. On the other hand, the strict transform of $P_o$ maps isomorphically to $P_o$,
but its normal bundle is now $\Ocal_{P_o}(-1)$. This implies that it is the exceptional divisor of blow-up of a smooth variety, so that we have a morphism  $\Bl_{H}(\Ps)\to \hat \Ps$  with $\hat \Ps$ smooth. We refer to the image of the contracted copy of $P_o$ as the \emph{vertex} and denote it by $v$.
The projection onto $\Delta$ subsists and yields a morphism $\hat \Ps\to \Delta$. This is another bundle of projective spaces (whose fiber over $o$ is the $\hat P_o$ above) and  we get a bimeromorphic map over $\Delta$:
\[
  f: \Ps/\Delta\leftarrow \Bl_{H}(\Ps)/\Delta\rightarrow \hat\Ps/\Delta.
\]
It is  clear that $\hat\Ps/\Delta$ coincides with $\Ps/\Delta$  over $\Delta-\{o\}$,  but that over $o$, the projective space $P_o$ is replaced by
the projective space $\hat P_o$. Also note that although $H$ has been contracted, it reappears as a hyperplane in $\hat P_o$.

These assertions can easily be verified by describing $f$ in coordinates: if  we are given a local trivialization of $\Ps$ at $o$ with coordinates $([T_0:\cdots : T_n],\lambda)$
such that  $H$ is defined by $T_0=0$ ($\lambda$ is the coordinate for $\Delta$), then we have a similar local trivialization
$([S_0:\cdots : S_n],\lambda)$ of $\hat \Ps$ at $o$ such that $f: \Ps\dashrightarrow\hat \Ps$ is given by
\[
  ([T_0: T_1:\cdots : T_n], \lambda)\mapsto ([T_0: \lambda T_1:\cdots : \lambda T_n], \lambda)=([T_0/\lambda: T_1:\cdots : T_n], \lambda).
\]
Note that the $S$-coordinates  vertex $v$ are $[1:0:\cdots :0]$.

Suppose that we are given an effective relative divisor $\Ds$ on  $\Ps/\Delta $ of degree $d$ such that $D_o=d H$. So in the above  coordinates, $\Ds$ is given up to first order by
$T_0^d+\lambda G(T_0, \dots , T_n)\pmod{\lambda^2}$ for some $G\in \CC[T_0, \dots, T_n]_d$.
If we make the base change over $\tau: \tilde \Delta\to \Delta$
given by $\tilde\lambda^d=\lambda$, then the relative divisor  $\tau^*\Ds$ on $\tau^*\Ps$ is given by $T_0^d+\tilde \lambda^d G(T_0, \dots , T_n)\pmod{\tilde\lambda^{2d}}$
and $\tau^*f: \tau^*\Ps\to \tau^*\hat\Ps$ takes $\tau^*\Ds$ to a relative divisor on $\tau^*\hat\Ps$ which is given by
\[
  S_0^d+ G(\tilde\lambda S_0, \dots , S_n)\equiv
  S_0^d+ G(0, S_1,\cdots ,S_n) \pmod{\tilde\lambda^d}
\]
(we substituted $T_0=\tilde\lambda S_0$, $T_i=S_i$ for $i=1, \dots, n$, and we divided by $\tilde\lambda^d$). We shall denote that divisor
simply by  $\hat\Ds$ (although something like $\hat\Ds(\phi)$ would be more appropriate). In particular, the fiber over $o$ is the hypersurface $\hat D_o$ defined by $S_0^d+ G(0, S_1,\cdots ,S_n)$. 
Note that it does not pass through the vertex $v=[1:0:\cdots :0]$. We put $G_o(S_1,\cdots ,S_n)=G(0, S_1,\cdots ,S_n)$. Observe that if $G_o$ is not identically zero, then its  zero set $G_o$ defines a hypersurface $Z\subset H$  and projection away from the vertex realizes $\hat D_o$ as a cyclic cover (with Galois group $\mu_d$) of the hyperplane $H$, which totally ramifies along $Z$. This $\mu_d$-action is the restriction of the natural $\mu_d$-action on $\tau^*\hat\Ps$ that is given by letting $\zeta\in\mu_d$ act as
\[
  ([S_0:S_1\cdots :S_n], \tilde\lambda)\mapsto ([\zeta S_0:S_1\cdots :S_n], \zeta\tilde\lambda).
\]
This action clearly preserves $\hat\Ds$.
\medskip

We apply this to the situation where we take for $\Ps$ the trivial bundle $\PP(\Sym^2U)\times\Delta\to \Delta$ and for $H$ the hyperplane
  {$H_K:=\PP(\Sym^4 V_K)$} in $\PP(\Sym^2U)\times\{ o\}$. On all these data we have $\Is$ acting, and indeed, on everything we do here this action
subsists. Let $j: \Xs:=\PP(U)\times \Delta\hookrightarrow \PP(\Sym^2U)\times \Delta$ be given by the Veronese embedding on the first factor and let
$\hat j: \hat \Xs \hookrightarrow \hat \Ps$ be its (strict) transform under the bimeromorphic map above. Since $\Xs$ meets $H_K\times\{ o\}$ transversally with preimage
$K\times \{ o\}$, we see that $\hat \Xs$ is obtained from $\Xs$ by blowing up $H_K\times\{ o\}$ followed by contraction of the strict transform
of $\PP(U)\times \{ o\}$. The effect of the bimeromorphic map $f_\Xs: \Xs\dashrightarrow\hat \Xs$ is that it replaces the projective plane  $X_o=\PP(U)$ by  the cone $\hat X_o$ over the Veronese image of $K$ in $H_K$.

We now think of the  Winger pencil $\Cs|\Delta$ as a divisor on $\Xs$ of relative degree $6$. It is the preimage under a divisor $\Ds$ on $\Ps$ of relative degree $3$ that has an equation of the form $T_0^3+\lambda G(T_0, \dots, T_5)$. Here $G$ is a cubic form on $\Sym^2U$  with
the property that precomposition with the squaring map $U\to \Sym^2U$ is the product $F$ of $6$ linear forms that define $C_\infty$ in $\PP(U)$.
So if we make the above base change for $d=3$, then $\hat \Ds$  is given by a divisor of relative degree $3$ whose fiber over $o$ is defined by $S_0^3+ G_o(S_1,\cdots ,S_5)$. This has the structure of a $\mu_3$-cover of $H_K$ ramified along the divisor defined by
$G_o$. The zero locus $Z$ of $G_o$ in $H_K$ meets $K$ (or rather its image under the  Veronese map) in the $\Is$-orbit of size $12$ (as a reduced divisor). Hence the strict transform of $\tau^*\Cs|\tilde\Delta$ under $f$ (which we denote simply by $\hat\Cs$) will has its fiber over $o$ the $\mu_3$-cover of $K$ which totally ramifies in this orbit. But this is just our
$C_K$ and the $\mu_3$-action is the restriction of a $\mu_3$-action on $\hat\Cs/\tilde\Delta$ which on the base is the obvious one.
It is clear that  $\Is$ acts fiberwise and commutes with this $\mu_3$-action. This modification of the degree $3$ base change of $\Cs|\Delta$
is in fact a universal deformation  of its closed fiber:

\begin{thm}\label{smfam}
  The central fiber of $\hat\Cs/\tilde\Delta$ is naturally identified with $C_K$ and the resulting deformation of $C_K$ is universal as a  deformation of a curve with $\Is$-action.
\end{thm}

\begin{proof}
  From the vanishing of $H^0(C_K,\theta_{C_K})$ and $H^2(C_K,\theta_{C_K})$ we know that the smooth universal deformation of $C_K$ as $\Acal_5$-curve exist, we will denote it by $\mathscr{U}_S\to S$. From the Theorems \ref{modspcursec} and \ref{modstkcur} below, the base of this family must have dimension 1, hence we can assume $S$ is an one dimensional open disk.
  Moreover the $\mu_3$-action on $C_K$ extends naturally to $\mathscr{U}_S\to S$ which will fix the central fiber and act nontrivially on $S$. Proposition \ref{howCkhapp} implies that this action is actually free on $S-\{o\}$. Now by the universal property of local universal deformation we can find a unique map $\phi:\tilde\Delta\to S$ such that $\hat{\Cs}/\tilde\Delta\cong \phi^{\ast}\mathscr{U}_S$. Let $e$ be  the degree of $\phi$.
  Then the definition implies that for any $\lambda\in\tilde\Delta-\{o\}$, we can find at least $3e$ many $\lambda'\in\tilde\Delta-\{o\}$ such that $\hat{C}_{\lambda}\cong\hat{C}_{\lambda'}$. But Corollary \ref{a5equiv} and our construction show that the only fibers isomorphic to $\hat{C}_{\lambda}$ are the three in its $\mu_3$-orbit. This implies that $e=1$. Hence $\hat\Cs/\tilde\Delta$ is universal.
\end{proof}

\section{Orbifold model and degeneration}

\subsection{\textbf{Coverings of genus zero orbifolds and their moduli}}\label{CGOM}
We will describe how to construct the coverings of a genus 0 orbifold with a prescribed (finite) Galois group and prescribed ramification data. This section is a summary of an article \textit{Geometry of the Wiman-Edge Pencil and the Wiman Curve} by B.~Farb and E.~Looijenga and we will explain in detail some of its  results that we will use. Suppose we have a descending sequence $p_1\geq\cdots \geq p_n$ of integers where $n\geq 4$ and $p_i\geq 2$. We will denote by $\vec{p}$ the tuple $(p_1,\cdots,p_n)$. Let $\pi_{0,n}$ denote the group generated by $a_1,\cdots,a_n$, subject to the only relation $a_1\cdots a_n=1$.

Let $\widetilde{\Mod}_{0,n}$ resp.\ $\widetilde{\Mod}(\vec{p})$  be the group consisting of automorphisms of $\pi_{0,n}$ which preserve  each conjugacy class of $a_i$ resp.\ permutes them in such a manner  that the conjugacy class of $a_i$ goes to the conjugacy class of $a_j$ if and only if $p_i=p_j$. These groups contain the inner automorphisms of $\pi_{0,n}$ as a normal subgroup and the quotients  $\Mod_{0,n}:=\widetilde{\Mod}_{0,n}/\Inn(\pi_{0,n})$
resp.\ $\Mod(\vec{p}):=\widetilde{\Mod}(\vec{p})/\Inn(\pi_{0,n})$ can be understood as  the mapping class group  of an $n$-pointed genus zero curve resp.\
of an orbifold of type  $(0;p_1,\cdots,p_n)$. It is clear that  $\Mod_{0,n}$ is a normal subgroup of $\Mod(\vec{p})$ with factor group the $\Scal_n$-stabilizer
of the function $i\mapsto p_i$ (denoted $\Scal(\vec p)$), so that we have a short exact sequence
\begin{equation}\label{exqModpModSnp}
  1\to \Mod_{0,n}\to \Mod(\vec{p})\to \Scal(\vec p)\to 1.
\end{equation}
These  groups also have an interpretation as (orbifold) fundamental group: $\Mod_{0,n}$ is the fundamental group of the fine moduli space $\Mcal_{0,n}$ of $n$-pointed genus zero curves (this is just the space of injective maps $\{1, \dots, n\}\to \PP^2$ modulo projective equivalance) and $\Mod(\vec{p})$ is the orbifold fundamental group of the  moduli space $\Mcal_0(\vec p)$ of orbifolds of type $(0; p_1, \dots, p_n)$. As will become clear below, the  latter is not a fine moduli space, but underlies a Deligne-Mumford stack (which we shall denote by $\MM_0(\vec p)$).

Let $G$ be a finite group with trivial centre,  for example $\Is$ or $\mathcal{S}_5$ and consider the set $\tilde{G}(\vec{p})$ of surjective group homomorphisms $\vec{g}:\pi_{0,n}\rightarrow G$ such that $\vec{g}(a_i)$ has order $p_i$.
Note that an element of $\tilde{G}(\vec{p})$ can also be considered as an ordered $n$-tuple $(g_1,\cdots,g_n)$ of generators of $G$ such that $g_i$ has order $p_i$ and $g_1\cdots g_n=1$. It is clear that  $\Aut(G)$ will act on the left of $\tilde{G}(\vec{p})$ by postcomposition. In particular,  $G$ then acts on $\tilde{G}(\vec{p})$ through inner automorphisms. In the tuple form, this action is just given by simultaneous conjugation. We will denote the quotient
$G\backslash \tilde{G}(\vec{p})$ by $G(\vec{p})$.
We let $\widetilde{\Mod}(\vec{p})$ act on the right of $\tilde{G}(\vec{p})$ by precomposition. Note that this action commutes with the $\Aut(G)$ action. It is also obvious that $\phi(\alpha x\alpha^{-1})=\phi(\alpha)\phi(x)\phi(\alpha)^{-1}$ for all $\phi\in\tilde{G}(\vec{p})$ and $\alpha,x\in\pi_{0,n}$ so that this induces a right-action of $\Mod(\vec{p})$ on ${G}(\vec{p})$.

Let $P$ be a copy of $\mathds{P}^1$. Suppose that we have an injective map $q:\{1,2,\cdots,n\}\rightarrow P$.  We put  $U_q=P-\{q_1,\cdots,q_n\}$. Choose a point $q_0\in U_q$, and  let $\alpha_i$ be a  simple loop which begins at $q_0$ and encircles $q_i$ in such a manner, that the loops
$\alpha_1, \dots, \alpha_n$ are like the petals of a flower.  Then an isomorphism  $\pi_{0,n}\cong \pi_1(U_q,q_0)$ is defined by  $a_i\leftrightarrow [\alpha_i]$.
So any $\vec{g}\in \tilde{G}(\vec{p})$  then determines a  surjective group homomorphism $\pi_1(U_q,q_0)\rightarrow G$.
By the theory of coverings, this defines  a connected  (unramified) $G$-covering $U_q(\vec{g})\to U_q$.
Moreover, for  elements $\vec{g}$ and $\vec{g'}$ of $\tilde{G}(\vec{p})$, $U_q(\vec{g})$ is $G$-isomorphic to $U_q(\vec{g'})$ if and only if
$\vec{g}$ and $\vec {g'}$ lie in the same $G$-orbit. We note that there is no non-trivial automorphisms of the $G$-covering $U_q(\vec{g})\to U_{q}$. For such automorphisms the property that it commutes with $G$-action means that this deck transformation lies in the center of $G$. Hence they are trivial since we assumed $G$ has trivial center.

Note that $U_q(\vec{g})$ is in a natural manner a complex curve  over $U_q$. By normalization we can extend this to a ramified
covering $P_q(\vec{g})\to P$. The covering space is  a projective nonsingular complex curve and $P_q(\vec{g})\rightarrow P$ is necessarily branched at $q_i$ with index $p_i$ (this is because the image of  $\alpha_i$ in $G$ has order $p_i$). For the same reason as above,  two such
coverings $P_q(\vec{g})\to P$ and $P_q(\vec{g'})\to P$ are $G$-isomorphic if and only if $\vec{g}$ and $\vec{g'}$ lie in the same orbit.
Moreover  no  covering $P_q(\vec{g})\to P$ has a nontrivial $G$-automorphism. So for a given $q$, the $G$-coverings of $P$ thus obtained  are in bijective correspondence with  $G(\vec{p})$. The absence of nontrivial automorphisms enables us to do this in families (by letting  $\vec q$ vary). In order to make this precise, we introduce two moduli functors.

Let us define the moduli functor $\Mk_{G,\vec{p}}^{M}$ of marked families as follows: given a scheme $S$ of finite type over $\CC$, let $\Mk_{G,\vec{p}}^{M}(S)$ be the set of isomorphism classes of tuple $(\phi:\mathscr{C}\to S,\Os_1,\cdots,\Os_n)$ where $\phi:\mathscr{C}\to S$ is a smooth family whose geometric fibers are connected curves endowed with a faithful $G$-action such that the orbifold quotient is of type $(0;p_1,\cdots,p_n)$ and $\Os_1,\cdots,\Os_n$ are sections of the family $\bar{\phi}: G\backslash \Cs\to S$ whose pre-image are irregular orbits with stabilizer $\mu_{p_i}$. Two such families $(\phi:\mathscr{C}\to S,\Os_1,\cdots,\Os_n)$ and $(\phi':\mathscr{C}'\to S',\Os_1',\cdots,\Os_n')$ are isomorphic if and only if we can find $G$-equivariant isomorphism between the families $\mathscr{C}\to S$ and $\mathscr{C}'\to S'$ which brings $\Os_i$ to $\Os_i'$.  We can also define the moduli functor $\Mk_{G,\vec{p}}$ of unmarked families by forgetting some information of the marked sections, i.e we allow the isomorphisms between two families can exchange $\Os_i$ and $\Os_j$ if $p_i=p_j$. Observe that for any such family $\phi:\Cs\to S$, if we restrict it to a small enough neighborhood, the local triviality of the smooth proper family implies that all its closed fibers must come from the same element in $G(\vec{p})$, i.e the only thing can be deformed in this family is the injective map $q$.

\begin{lem}\label{noauto}
  If a  $G$-curve has $G$-quotient of the type $(0;\vec{p})$, then every $G$-automorphism of that curve has finite order. {In particular,  the only $G$-automorphism that preserves each irregular $G$-orbit is the identity}.
\end{lem}
\begin{proof}
  To prove the first claim, suppose $\psi$ be any $G$-automorphism of $C$. Then the $G$-quotient of $\psi$ may exchange the sections $\Os_i$ and $\Os_j$ if $p_i=p_j$. Let $o_i$ still be the image of $\Os_i$ in $G\backslash C\cong \PP^1$ and $\bar{\psi}$ be the induced map of $\psi$ on $G\backslash C$. Now there must exist a positive integer $m$ which is determined only by $\vec{p}$, such that $\bar{\psi}^m$ has each $o_i$ as one of its fixed point. Since $n\geq 4$, $\bar{\psi}^m$ must be identity and $\psi^m$ is a deck transformation. Since $\Aut(G)$ has no center, all the deck transformations must be trivial.

  For the second claim, if we have one such automorphism $\psi$ of $C$ which fixes the orbital $G$-sections, then $m=1$ and $\psi$ itself is a deck transformation. This finishes the proof.
\end{proof}
\begin{thm}\label{finemod}
  The moduli functor $\mathfrak{M}_{G, \vec{p}}^M$ is represented by a fine moduli space $BG^M(\vec{p})$, which is quasi-projective variety with
  $\sharp (G(\vec{p})/\Mod_{0,n})$ many connected components. By assigning to an element of  $\mathfrak{M}_{G, \vec{p}}^M(S)$  its $G$-quotient, we get a natural morphism $BG^M(\vec{p})\to \Mcal_{0,n}$ which is a finite cover of degree $\sharp G(\vec{p})$.
\end{thm}
\begin{proof}
  By considering all the possibilities of $q$, we obtain a covering map from the set of all isomorphism types of covering space mentioned above to the configuration space $P^{(n)}$ of injective maps $q:\{1,\cdots,n\}\rightarrow P$. This defines a diagram
  \begin{equation}\label{confspcov}
    \tilde P^{(n)}\to  P^{(n)}
  \end{equation}
  This is an unramified covering of degree $\# G(\vec{p})$. We want  to lift the action of  $\Aut(P)\cong \PSL_2(\CC)$ to this diagram. For this we must a priori  pass to its universal cover $\widetilde\Aut(P)\cong \SL_2(\CC)$. This is a  $\mu_2$-cover. By Lemma \ref{noauto} it is clear that  $\mu_2\subset\widetilde\Aut(P)$ acts trivially on $ \tilde P^{(n)}$. So the  free  $\Aut(P)$-action on $P^{(n)}$ lifts to the morphism (\ref{confspcov}). If we  divide out  by this action (which  we can also implement by  fixing  $q_1,q_2,q_3$ to be  $0,1,\infty$ and varying  the other points), we obtain a morphism of algebraic varieties
  \begin{equation}
    {B}G^M(\vec{p})\rightarrow \Mcal_{0,n},
  \end{equation}
  which is still an unramified covering of degree $\# G(\vec{p})$.

  We claim the variety $BG^M(\vec{p})$ is the fine moduli space of the moduli functor $\mathfrak{M}_{G, \vec{p}}^M$. To see this we only need to construct the universal family over $BG^M(\vec{p})$.

  Given a smooth $G$-curve $C$ and a collection of irregular orbits $\{O_1,\cdots,O_n \}$ as before, we assume its $G$-quotient
  $G\backslash C$ is represented by the point $q\in \Mcal_{0,n}$ and $C$ is recovered from $G\backslash C$ by
  $\vec{g}\in G(\vec{p})$. Given a small open neighborhood  {$U\subset \Mcal_{0,n}$} of $q$, we have an universal deformation
  $(\mathscr{P}_U\to U,\mathscr{O}_{1,U},\cdots,\mathscr{O}_{n,U})$ of $G\backslash C$ which is a family of $n$-marked projective lines and
  $\Os_{i,U}:U\to \mathscr{P}_U$ are sections of the family. The construction described above gives a deformation
  $(\mathscr{U}_{\tilde{U},\vec{g}}\to \tilde{U},\Os_{1,\tilde{U}},\cdots,\Os_{n,\tilde{U}})$ of $C$ where $\tilde{U}$ is isomorphic to $U$ and
  $\Os_{i,\tilde{U}}:\tilde{U}\to G\backslash\Us_{\tilde{U}}$
  are sections associated to $\Os_{i,U}$. Note that this means $\tilde{U}$ can be treated as a open neighborhood of $[C]\in BG^M(\vec{p})$. We will denote such families by only $\mathscr{U}_{\tilde{U},\vec{g}}\to \tilde{U}$ if no ambiguity happens.

  It is clear that this family $(\mathscr{U}_{\tilde{U},\vec{g}}\to \tilde{U},\Os_{1,\tilde{U}},\cdots,\Os_{n,\tilde{U}})$ is an universal deformation of smooth algebraic curve with $G$-action. Moreover the construction implies that
  $\mathscr{U}_{\tilde{U},\vec{g}}\to \tilde{U}$ is also universal for its nearby fibres.

  These universal deformations actually glue into a family $(\mathscr{U}_{BG^M(\vec{p})}\to BG^M(\vec{p}),\Os_1,\cdots,\Os_n)$. For if we have two such
  local deformations $\mathscr{U}_{\tilde{U},\vec{g}}\to \tilde{U}$ and
  $\mathscr{U}_{\tilde{U}',\vec{g}}\to \tilde{U}'$ with another $G$-curve
  $[C'',O_1'',\cdots,O_n'']\in \tilde{U}\cap \tilde{U'}$. Since two deformations are both universal for generic fibers, they are both universal
  deformations for $(C'',O_1'',\cdots,O_n'')$. This implies
  $\mathscr{U}_{\tilde{U},\vec{g}}\to {\tilde{U}\cap\tilde{U'}}$ and
  $\mathscr{U}_{\tilde{U'},\vec{g}}\to {\tilde{U}\cap\tilde{U'}}$ can be
  identified together and the absence of automorphism group implies this identification map is unique. By gluing all such universal deformations
  in $BG^M(\vec{p})$, we get a family $(\mathscr{U}_{BG^M(\vec{p})}\to BG^M(\vec{p}),\Os_1,\cdots,\Os_n)$.

  This is also an universal family. To see this, let us consider any other smooth family of $G$-curves $(\mathscr{C}_B\to B,\Os_{1,B},\cdots,\Os_{n,B})$. For any $t_0\in B$ and a small open neighborhood $\tilde{V}_{t_0}$ of $t_0$ in $B$,
  the restriction $\mathscr{C}_B\to {\tilde{V}_{t_0}}$ is a deformation of
  $\Cs_{t_0}$ and it's irregular orbits. Hence we have an unique locally defined map $f_{t_0}:\tilde{V}_{t_0}\to \tilde{U}_{t_0}$ such that $(\mathscr{C}_B\to B,\Os_{1,B},\cdots,\Os_{n,B})|_{\tilde{V}_{t_0}}\cong f_{t_0}^{\ast}(\mathscr{U}|_{\tilde{U}_{t_0}}\to\tilde{U}_{t_0},\Os_{1,\tilde{U}_{t_0}},\cdots,\Os_{n,\tilde{U}_{t_0}})|_{\tilde{U}_{t_0}}$ where $\mathscr{U}|_{\tilde{U}_{t_0}}\to\tilde{U}_{t_0}$ is the universal deformation of $\Cs_{t_0}$ defined above. Same reason as above these $f_{t_0}$ glue to a morphism $f:B\to BG^M(\vec{p})$ such that $f^{\ast}\mathscr{U}_{BG^M(\vec{p})}\cong \mathscr{C}_B$ and such map is unique. The number of connected component comes from the fact that $\Mcal_{0,n}$ is connected with fundamental group $\Mod_{0,n}$. This finishes our proof.
\end{proof}
\begin{thm}\label{thm:orbifoldmoduli}
  The functor  $\mathfrak{M}_{G, \vec p}$ is given by a Deligne-Mumford stack $\BB G(\vec{p})$ whose underlying coarse moduli space $BG(\vec{p})$ is a quasi-projective variety with $\sharp (G(\vec{p})/\Mod(\vec{p}))$ connected components.
  By assigning to an element of  $\mathfrak{M}_{G, \vec p}(S)$  its $G$-quotient, we get a natural morphism  $\BB G(\vec{p})\to \MM_0(\vec p)$ of stacks resp.\ $BG(\vec{p})\rightarrow \mathcal{M}_0(\vec p)$ of varieties. This a finite cover of degree $\sharp G(\vec{p})$.
\end{thm}

\begin{proof}
  Theorem \ref{finemod} gives the diagram
  \begin{equation}\label{BGMhat}
    \mathscr{U}_{{B}G(n)}\to  {B}G^M(\vec{p})\rightarrow \Mcal_{0,n},
  \end{equation}
  The first arrow defines an element of $\mathfrak{M}_{G,\vec{p}}({B}G^M(\vec{p}))$ and as we have proved it can be regarded as  a fine moduli space for the tuple $(C;O_1,\cdots,O_n)$, where $C$ is a smooth projective $G$-curve, whose $G$-quotient is an orbifold of type $(0;p_1, \dots, p_n)$
  and irregular $G$-orbits  have been numbered  $(O_1,\cdots,O_n)$ such that $O_i$ has size $ |G|/ p_i$.

  The modular interpretation of diagram (\ref{confspcov}) shows that the obvious action of the  finite group $\Scal(\vec p)$ on $ P^{(n)}$ (which permutes the factors) also lifts to this diagram. This commutes with the action of $\Aut(P)$, but the product action of $\Aut(P)\times \Scal(\vec p)$ may have  nontrivial isotropy in $ P^{(n)}$, as some nontrivial element in $\Aut(P)$  could permute the ramification points  $(q_1, \dots , q_n)$ in a way as to preserve their weights.
  Therefore from Lemma \ref{noauto} the  quotient by this action is a Deligne-Mumford stack
  \begin{equation}\label{BGM}
    \Cs_{\BB G (\vec{p})}\to \BB G(\vec{p})\rightarrow \MM_0 (\vec p)
  \end{equation}
  Its modular interpretation is that of (\ref{BGMhat}), except that there is no numbering of the irregular orbits with the same type. The second map of (\ref{BGM}) is a finite cover whose underlying morprhism of coarse moduli space $BG(\vec{p})\rightarrow \Mcal_0(\vec p)$ is the one appearing in the theorem.
  A  fiber of  $BG(\vec{p})\rightarrow \Mcal_0(\vec p)$  over a non-orbifold point is  as a $\Out(G)\times \Mod(\vec{p})$-set identified with $G(\vec{p})$. In particular, since $\Mod(\vec{p})$ is the fundamental group of $\Mcal_0(\vec{p})$, the connected components of $BG(\vec{p})$ are indexed by $G(\vec{p})/\Mod(\vec{p})$.
  The assertions of the theorem now follow.
\end{proof}

\begin{rmk}\label{rmk:DM}
  We have to resort to stacks, because there might exist some $G$-curve admitting an automorphism which nontrivially permutes  its irregular orbits. This is indeed happening in our situation. On the other hand  if  this does not happen, then we have indeed  a diagram of varieties
  \begin{equation*}\label{BGMhatv}
    \Cs_{BG(\vec{p})}\to BG(\vec{p})\rightarrow \Mcal_0(\vec p)
  \end{equation*}
  In particular, $\mathfrak{M}_{G, \vec p}$ has then a fine moduli space.
\end{rmk}

\subsection{\textbf{Moduli space of smooth projective genus 10 curve with faithful $\mathcal{A}_5$ Action}}

\begin{prop}\label{quta55222}
  Suppose $C$ is a smooth genus 10 curve endowed with a faithful $\mathcal{A}_5$-action. Let  $f:C\rightarrow P$ form the quotient by this action. Then $P$ is an orbifold of type $(0;5,2,2,2)$.
\end{prop}
\begin{proof}
  Obviously $P$ is an orbifold. By the Riemann-Hurwitz formula, we have
  \begin{equation*}
    2g(C)-2=\deg(f)(2g(P)-2)+\sum_{X\in C}(e_X-1)
  \end{equation*}
  Here $\deg(f)=\sharp\mathcal{A}_5=60$, and $e_X$ denotes the ramification index of $f$ at $X$. Since $e_X=1$ for all but finite many points, thus we may write it as $\sum_{Y\in R\subset P}\sum_{X\in f^{-1}Y}(e_X-1)$ and $R$ is the ramification locus. Denote by $G_X\subset \mathcal{A}_5$ be the stabilizer of $X\in C$.  Since $\mathcal{A}_5$ acts on each fiber transitively and $\sum_{X\in f^{-1}Y}e_X=60$, we have $\alpha(Y):=\sum_{X\in f^{-1}Y}(e_X-1)=(60-\frac{60}{\sharp\ G_X})\ge0$. This gives the equation
  \begin{equation*}
    138=120g(P)+\sum_{Y\in R}\alpha(Y).
  \end{equation*}
  By Lemma \ref{faithstab}, a stabilizer must be cyclic, and so $\alpha(Y)$ can take only the values listed in Table \ref{tva}.
  \begin{table}
    \centering
    \caption{The value of $\alpha$ according to stabilizer}\label{tva}
    \begin{tabular}{llll}
      \hline\noalign{\smallskip}
      conjugacy class of a generator & Order & Value of $\alpha$ \\
      \noalign{\smallskip}\hline\noalign{\smallskip}
      (1)                            & 1     & 0                 \\
      (12)(34)                       & 2     & 30                \\
      (123)                          & 3     & 40                \\
      (12345)                        & 5     & 48                \\
      \noalign{\smallskip}\hline
    \end{tabular}
  \end{table}
  A computation then shows that the only solution is: $g(P)=0$ with 4 orbifold  points, 3 of which have stabilizer of order 2,  and one having order 5.
  This proves our claim.
\end{proof}

We will use the methods in the last section to describe the moduli space of such curves, where we take $G=\Acal_5$ and $\vec p=(5,2,2,2)$.
Note that in that case $\Scal(\vec p)\subset \Scal_4$ is then the permutation group of the last three items (so that is a copy of $\Scal_3$).

We begin with several combinatorial lemmas.
\begin{lem}\label{g1g2}
  Let $\mathcal{A}_5(r)$ denote the set of elements of $\mathcal{A}_5$ of order $r$. Consider the $\mathcal{A}_5$-action on $\mathcal{A}_5(5)\times\mathcal{A}_5(2)$ by simultaneous conjugation. Then this action is free, and every orbit is represented by one of the following 6 pairs:
  \begin{description}
    \item[$ord(g_1g_2)=2$] $(g_1,g_2)$ equals $((12345),(12)(35))$ or $((12354),(12)(34))$,
    \item[$ord(g_1g_2)=3$]
      $(g_1,g_2)$ equals $((12345),(12)(34))$ or $((12354),(12)(45))$,
    \item[$ord(g_1g_2)=5$] $(g_1,g_2)$ equals $ ((12345),(13)(25))$ or $((12354),(13)(25))$.
  \end{description}
\end{lem}
\begin{proof}
  First we prove that the $\mathcal{A}_5$-action is free. If that is not is case, then we can find $a\in\mathcal{A}_5(5)$, $b\in\mathcal{A}_5(2)$ and $g\in\mathcal{A}_5$ such that $g^{-1}ag=a$ and $g^{-1}bg=b$. Without loss of generality, we may assume $a=(12345)$. Then we have $g^{-1}ag=(g(1),g(2),g(3),g(4),g(5))$. But $(12345)$ has only finite many forms, list them all we will find that $g$ must be a power of $a$. But such element can not fix $b$ under conjugation. Since for the same reason $b$ must be a power of $g$, thus a power of $a$, a contradiction. Now we have $12\times 30=360$ elements in $\mathcal{A}_5(5)\times \mathcal{A}_5(2)$ and $\mathcal{A}_5$ has 60 elements and $\mathcal{A}_5$-action is free. Thus we must have 6 different orbits.

  To list all the orbits we may take a proper element $g\in \mathcal{A}_5$ such that $g^{-1}g_1g$ is one of two following elements $(12345)$ or $(12354)$. In this case $g^{-1}g_2g$ can be any order 2 element. By calculation, we find that all possibilities are as above.
\end{proof}
\begin{lem}\label{g3g4}
  Take $h\in \mathcal{A}_5(r)$ and consider the set of $(h_1,h_2)\in \mathcal{A}_5(2)\times\mathcal{A}_5(2)$ such that $h=h_1h_2$. Then for $r=2,3,5$ this set has $r$ elements. For $r=2$, $h_1$ and $h_2$ commute. For $r=3$ or $r=5$, this is a free $\langle h\rangle$-orbit under simultaneous conjugation.
\end{lem}
This Lemma comes from computation.
\begin{prop}\label{g1g2g3g4}
  The tuples $(g_1,g_2,g_3,g_4)\in\mathcal{A}_5(5,2,2,2)$ come in three types, according to the order $r$ of $g_1g_2$. For $r=2$ we have 4 elements, for $r=3$ we have $6$ elements and for $r=5$ we have 10 elements, giving therefore 20 elements in total.

  The group $\Mod_{0,4}$ acts transitively on the subset of $\mathcal{A}_5(5,2,2,2)$ whose first coordinate is conjugate to  $(12345)$ resp.\   $(12354)$, so that $\Mod_{0,4}$ has two orbits in $\Acal_5(5,2,2,2)$.

  The group $\Mod(5,2,2,2)$ acts transitively on the subset of $\mathcal{A}_5(5,2,2,2)$ whose first coordinate is conjugate to  $(12345)$ resp.\   $(12354)$, so that  $\Mod(5,2,2,2)$ has two orbits in $\mathcal{A}_5(5,2,2,2)$.
\end{prop}
\begin{proof}
  The first statement comes from the last two lemmas. First by Lemma \ref{g1g2}, we may assume $(g_1,g_2)$ is one of the 6 pairs listed there. Then $g=g_1g_2$ has order $r$ where $r=2,3$ or 5. Lemma \ref{g3g4} implies that if $(g_3,g_4)$ and $(g'_3,g'_4)$ lie in the same orbit, then they can be transformed into each other through  conjugation by a some power of $g$. But similarly as above, such a conjugation cannot fix the first two coordinates. So the first statement follows.

  For the second and the third statements, let us take any such $g=(g_1,g_2,g_3,g_4)$ where $g_1=(12345)$. First of all, by Lemma \ref{g3g4}, if we fix the order of $g_1g_2=(12345)g_2$, then $\Mod_{0,4}$ acts transitively inside these `types'. The transformation map is given by
  \begin{equation*}\label{morptrans1}
    (a_1,a_2,a_3,a_4)\rightarrow(a_1,a_2,(a_1a_2)a_3(a_1a_2)^{-1},(a_1a_2)a_4(a_1a_2)^{-1})
  \end{equation*}
  To show that $\Mod_{0,4}$ acts transitively on `types', we only need to work on some examples. We first let $g$ is the element $((12345),(12)(35),(15)(34),(14)(35))$. Here we find $g_1g_2$ has order 2. Then the map
  \begin{equation*}
    (a_1,a_2,a_3,a_4)\rightarrow (a_2^{-1}a_1a_2,a_3a_2a_3^{-1},a_3,a_2^{-1}a_4a_2)
  \end{equation*}
  will map $g$ to $((15432),(14)(25),\ast,\ast)$. By calculation $(15432)(14)(25)=(13245)$ which has order 5. Next we let $g$ to be the element $((12345),(13)(25),(12)(34),(24)(35))$. In this case $g_1g_2=(12345)(13)(25)$ is of order 5. This time the same map has image $((14235),(15)(24),\ast,\ast)$. By calculation $(14235)(15)(24)=(354)$ which has order 3. Since $\Mod_{0,4}$ is a subgroup of $\Mod(5,2,2,2)$, the same results holds for $\Mod(5,2,2,2)$.

  On the other hand, by the definition of $\Mod_{0,4}$ resp. $\Mod(5,2,2,2)$, the conjugacy class of $a_1$ must be fixed. Thus $g_1$ must lie in the conjugacy class of $(12345)$ (not of $(12354)$). In summary, we have 20 elements in $\mathcal{A}_5(5,2,2,2)$ which make up two $\Mod_{0,4}$ resp. $\Mod(5,2,2,2)$-orbits of size 10.
\end{proof}
When $g_1=(12345)$ we made a table of all values of $\Acal_5(5,2,2,2)$ see Table \ref{tvg1234}.
\begin{table}
  \centering
  \caption{Values of $(g_1,g_2,g_3,g_4)$}\label{tvg1234}
  \begin{tabular}{lllll}
    \hline\noalign{\smallskip}
    Order of $g_1g_2$ & $g_1$   & $g_2$    & $g_3$    & $g_4$    \\
    \noalign{\smallskip}\hline\noalign{\smallskip}
    2                 & (12345) & (12)(35) & (15)(34) & (14)(35) \\
    2                 & (12345) & (12)(35) & (14)(35) & (15)(34) \\
    3                 & (12345) & (12)(34) & (15)(24) & (24)(35) \\
    3                 & (12345) & (12)(34) & (13)(24) & (15)(24) \\
    3                 & (12345) & (12)(34) & (24)(35) & (13)(24) \\
    5                 & (12345) & (13)(25) & (14)(25) & (15)(23) \\
    5                 & (12345) & (13)(25) & (12)(34) & (24)(35) \\
    5                 & (12345) & (13)(25) & (13)(45) & (12)(34) \\
    5                 & (12345) & (13)(25) & (15)(23) & (13)(45) \\
    5                 & (12345) & (13)(25) & (24)(35) & (14)(25) \\
    \noalign{\smallskip}\hline
  \end{tabular}
\end{table}

\begin{prop}\label{howCkhapp}
  There are only two smooth algebraic curves {of genus 10} endowed with a faithful $\Acal_5$-action up to an $\Acal_5$-isomorphism which admit an $\Acal_5$-automorphism of order three. Moreover these are the only possible smooth $\Acal_5$-curves with non-trivial $\Acal_5$-automorphisms.
\end{prop}
\begin{proof}
  We have showed in Theorem \ref{tricov} that there are at least two different smooth $\Acal_5$-curves $C_K$ resp. $C'_K$ admitting an $\Acal_5$-automorphism of order three. In fact, by quotient out these $\Acal_5$-automorphisms, we get a smooth conic $K$ resp. $K'$ with faithful $\Acal_5$-action. Moreover $K$ and $K'$ are all the possibilities that $\Acal_5$ acts on a copy of $\PP^1$.

  We only need to show they are all possible cases. If not, let $C$ be a smooth curve of genus 10 endowed with a faithful $\Acal_5\times \mu_3$-action. Then Riemann-Hurwitz formula and Lemma \ref{faithstab} implies that the smooth curve $\mu_3\backslash C$ has two possibilities. One is of genus 0 with 12 ramification points which is isomorphic to the known cases by the uniqueness part of Theorem \ref{tricov}. The other one is of genus 4 with no ramifications. However in this case Remark \ref{norS5act} implies that $\mu_3\backslash C$ has two irregular orbits of size 12 and one irregular orbit of size 30 which contradicts to Proposition \ref{quta55222}.

  To show the last assertion, note that by Proposition \ref{quta55222}, we only need to show that we don't have a smooth $\Acal_5$-curve which admits an $\Acal_5$-automorphism cyclic of order two. If not, assume that $C$ is a smooth $\Acal_5$-curve which admits an $\Acal_5$-automorphism $\tilde{\psi}$ of order two. Let $\psi$ be the map induced from $\tilde{\psi}$ by quotient out the $\Acal_5$-action. Without loss of generality, we may assume the four orbifold points on $P$ are $\infty,-1,1$ and $0$ and the map $\psi$ is given by $z\to -z$. Let us concentrate on the point $0$ which is an orbifold point of order 2 and fixed by $\psi$. Under suitable local coordinate, the quotient map $f:C\to P$ is given by $w\to z=w^2$. The equivariant condition implies $\tilde{\psi}$ is of the form $w\to i w$ or $w\to -iw$. In particular it is of order four. A contradiction!

\end{proof}

\subsection{\textbf{Degeneration}}
Now we will allow the two of the order 2 orbifold points come together into one point (which we shall call the \emph{special point}). This defines an $\mathcal{A}_5$-cover $C'$ of $P'$ where $P'$ is an orbifold of genus 0 and 3 orbifold points. We  will study $C'$ form the property of curves in the last section.

Let $P$ be an orbifold of type $(0;5,2,2,2)$. The orbifold points are named as $z_1,z_2,z_3,z_4$. And $C$ is smooth genus 10 curve with faithfully $\mathcal{A}_5$-action. The morphism $f:C\rightarrow P$ is canonical map by quotient $\mathcal{A}_5$ which is also a ramified $\mathcal{A}_5$-covering of $P$ branched at 4 orbifold points. Let us take a open disk $D\subseteq P$ which contains $z_3$ and $z_4$, but not $z_1$ and $z_2$, so that $f^{-1}D\to D$ is an $\mathcal{A}_5$-cover ramified over  $z_3$ and $z_4$ only. Let $\gamma$ be a simple arc in $D$ that joins $z_3$ with $z_4$.  Let $D'\subset f^{-1}D$ be a connected component of $f^{-1}D$ and  $\gamma'$ be a connected component of $f^{-1}\gamma\cap D'$. We have the following Lemma:
\begin{lem}\label{primplygn}
  The curve $\gamma'$ is a polygon which is  homotopic to a \emph{nonseparating} embedded circle on $C$
  (this means that its complement is connected). The number of its edges is give by $2\Ord(g_3g_4)$ where $g_3$ and $g_4$ are the same as the one defined in Proposition \ref{g1g2g3g4}.
\end{lem}
\begin{proof}
  To see the first assertion, note first that $\gamma$ is the deformation retract of $D$. Since $f$ is an ramified $\mathcal{A}_5$-cover over $D$ with ramification in the end point $z_3$ and $z_4$ of $\gamma$ only,  a deformation retraction of $D$ onto $\gamma$ lifts (uniquely) to a deformation retraction of $f^{-1}D$ onto $f^{-1}\gamma$. Since the map $f:f^{-1}\gamma\to\gamma$ has simple ramification only over the end points of  $\gamma$, every connected component will be a polygon (so homeomorphic to a circle) with the vertices being the preimage of $z_3$ and $z_4$.

  To see the number of its edges, we will working on the monodromy of $f:C\to  P$. Let us take a generic point $x\in D$ and fix a lifting $x'\in f^{-1}x$. We will have a monodroy map $\rho_{x'}: \pi_1(P-\{z_1,z_2,z_3,z_4\},x)\to \Acal_5$ as following: since $f$ is an $\Acal_5$-cover  after removing $\{z_1,z_2,z_3,z_4\}$, there exists an unique element $g_l$ for any $l\in\pi_1(P-\{z_1,z_2,z_3,z_4\},x)$ such that $g_l.x'$ is the ending point of the lifting of $l$ whose beginning point is $x'$. We will let  $\rho_{x'}(l)=g_l$. Clearly the uniqueness of the lifting implies that $\rho_{x'}(l_1l_2)=\rho_{x'}(l_1)\rho_{x'}(l_2)$ and $\rho_{g.x'}=g\rho_{x'}g^{-1}$.

  Let $s_3$ resp. $s_4$ a simple loop in $D\subset P$ based at $x$ encircles $z_3$ resp. $z_4$ only. The curve $s$ is a simple loop base at $x$ and homotopic to $\partial D$ in $D-\{z_3,z_4\}$. Without loss of generality, we will assume $s_3$, $s_4$ and $s$ are oriented counterclockwise. Since $f^{-1}\gamma$ is the deformation retract of $f^{-1}D$, the number of its edges is $2\Ord(\rho_{x'}(s))$. The latter equals to $2\Ord(\rho_{x'}(s_3)\rho_{x'}(s_4))=2\Ord(g_3g_4)$.

  To see it is nonseparating, suppose otherwise so that $C-\gamma'$ is not connected.  It is clear that  $f^{-1}\gamma$ resp. $f^{-1}(C-\gamma)$ is  $\Acal_5$-invariant and $\Acal_5$ acts transitively on $\pi_0(f^{-1}\gamma)$ resp. $\pi_0(C-f^{-1}\gamma)$. Since $\gamma'$ is separating, the same is the other connected component of $f^{-1}\gamma$. Hence we have  the following
  \begin{equation*}
    \sharp \pi_0(C-f^{-1}\gamma)=\sharp \pi_0(f^{-1}\gamma)+1=\frac{60}{\sharp \hbox{Stabilizer of $\gamma'$}}+1
  \end{equation*}
  However combining the Proposition \ref{g1g2} and the proof above implies that this is impossible.
\end{proof}

\begin{prop}\label{nodalcurve}
  The curve  $C'$ is a nodal curve of arithmetic genus 10 whose nodes lie over the special point.
\end{prop}
\begin{proof}
  The Lemma \ref{primplygn} imply that this polygon is a vanishing locus for the  degeneration: it gets contracted to produce a node above the special point.
\end{proof}

\begin{rmk}
  If we consider the GIT quotient of the projective space of effective degree 11 divisors on $\PP^1$ by $\SL_{2}(\mathds{C})$, then for the obvious
  $\SL_{2}(\mathds{C})$-linearisation, a divisor is stable if only if it has no point of multiplicity $>5$ and  here semistability is equivalent to stability.
  A point $(z_1,z_2,z_3,z_4)\in\mathcal{M}_{0,4}$ with weight $(5,2,2,2)$ can be considered as a stable degree 11 divisor on  $\PP^1$ if
  weight is interpreted as multiplicity. So if we allow only two of $z_2,z_3,z_4$ to be equal, then we still get a stable point, but any other
  type of coalescing yields an unstable point. This explains why we only allow two order 2 points to coincide. For more details on this quotient see Chapter 2 of \cite{dolgachev1988point}.
\end{rmk}

Let $\tilde{C}'\to C'$ be the normalization of $C'$. By its universal property, the $\Acal_5$-action on $C'$ lifts to $\tilde{C}'$. Note that $C'$ is not necessarily irreducible, and so $\tilde{C}'$ might be disconnected.
\begin{prop}\label{prop:stab}
  Let $y\in P$ be the special point. Then for every $x\in \tilde{C}'$ over $y$, the $\Acal_5$-stabilizer of $x$ is cyclic of order $n\in\{ 2, 3,5\}$.
  In particular, the special point is an orbifold point of type $n$.
\end{prop}
\begin{proof}
  Denote by $\tilde C'_i$ the connected component of $\tilde C'$ which contains $x$. No nontrivial element of $\Acal_5$ can act as the identity
  on $\tilde C'_i$ and so the $\Acal_5$-stabilizer of $\tilde C'_i$ acts faithfully on $\tilde C'_i$.
  Lemma \ref{faithstab} then implies that the $\Acal_5$-stabilizer of $x$ is cyclic. This stabilizer cannot be trivial and  so the proposition then follows from the fact that the  elements of $\Acal_5-\{1\}$ have order
  $2$, $3$ or $5$.
\end{proof}

\begin{rmk}
  We can distinguish these three cases using the polygon mentioned in the Lemma \ref{primplygn}. In the order 2 case, we have four points with two in the fiber of $z_3$ and two in the fiber of $z_4$. Each point connects with the other two and no two points lie in the same fiber are connected. So the polygon  is a square. In the order 3 case, we have six points, three lie above $z_3$ and three lie above $z_4$ and we get a hexagon. Finally in the order 5 case, we have ten points, five lie above $z_3$ and five above $z_4$ and we get a decagon.
\end{rmk}

\begin{prop}\label{deg}
  Let  $n$ be defined as in Proposition \ref{prop:stab}. Then
  \begin{description}
    \item[$n=2$]  $\tilde{C}'$ is the disjoint union of 6 lines and $C'$ is the union of 6 lines, each two such lines intersect, but no three pass through the same point.
    \item[$n=3$] $\tilde{C}'$ is irreducible and is of genus 0 and $C'$ is irreducible with 10 nodes.
    \item[$n=5$] $\tilde{C}'$ is irreducible and is of genus 4 and $C'$ is irreducible with 6 nodes.
  \end{description}
\end{prop}
\begin{proof}
  Consider the dual graph $\Gamma$ of $C$. This is the graph whose vertices are indexed by the irreducible component of $C'$ and whose edges are indexed the nodes of $C'$ with an edge connecting the vertices defined by the irreducible components on which the node lies (for more details, see \cite{arbarello2011geometry} Chapter X Section 2).

  Now we have the equation that
  \begin{equation*}
    p_a(C')=\sum p_a(\tilde{C}'_i)+1-v+e
  \end{equation*}
  where the sum is over the connected components $\tilde{C}'_i$ of $\tilde{C}'$, $v$ is the number of vertices of $\Gamma$ which is also the number of irreducible components of $C'$ and  $e$ is the number of edges of $\Gamma$ which equals to the number of nodes of $C'$.
  Since $\tilde{C}'\to P$ is a $\Acal_5$-cover, $\Acal_5$ acts transitively on the set of connected components $\tilde{C}'_i$ of $\tilde{C}'$, so that
  $p_a(\tilde{C}'_i)$ is in fact independent of $i$.
  This also implies that  the number $v$ of connected components must be the index of some subgroup of $\mathcal{A}_5$ in $\mathcal{A}_5$.
  It follows that
  \begin{equation*}
    v(p_a(\tilde{C}'_i)-1)=p_a(C')-1-e=9-e
  \end{equation*}
  For the value of $e$, by the last proposition, we know that it must has only possibilities of 15, 10 and 6 (nodes must lie in the fiber of $y$ and normalization makes one nodes into two points). So the possible solutions are as follows:
  \begin{description}
    \item[$e=15$, $v=6$, $p_a(\tilde{C}')=0$] in this case $\tilde{C}'$ is the disjoint union of 6 lines, $C'$ is the union of 6 lines, every two intersecting, but no three passing through the same point and  the $\Acal_5$-orbit space  is an orbifold of type $(0;5,2,2)$.
    \item[$e=10$, $v=1$, $p_a(\tilde{C}')=0$] then $\tilde{C}'$ is a projective line, then $C'$ is irreducible curve with 10 nodes and  the $\Acal_5$-orbit space is an orbifold of type $(0;5,2,3)$
    \item[$e=6$, $v=1$, $p_a(\tilde{C}')=4$] then $\tilde{C}'$ is a genus 4 curve, $C'$ is irreducible with 6 nodes and the $\Acal_5$-orbit space is an orbifold of type $(0;5,2,5)$.
  \end{description}
  A priori  we might also have $e=6$, $v=3$. Then  $\tilde{C}'$ has 3 connected components, but since $\mathcal{A}_5$ has no subgroup of order 20, this case is impossible.

  This ends the proof.
\end{proof}
We can now  prove our main theorems.

\begin{thm}\label{modspcursec}
  { We have a fine moduli space $B\Acal_5^M(5,2,2,2)$ of the tuple $(C,O_1,O_2,O_3,O_4)$ where $C$ is a smooth quasi-projective genus 10 curves endowed with a faithful $\Acal_5$-action and tuple $(O_1,O_2,O_3,O_4)$ are marked irregular orbits. The space $B\Acal_5^M(5,2,2,2)$ has 2 connected components, each of which is a degree 10 cover of $\Mcal_{0,4}$.}
\end{thm}
\begin{thm}\label{modstkcur}
  {We have a moduli stack $\ \BB\Acal_5(5,2,2,2)$ of smooth quasi-projective genus 10 curves endowed with a faithful $\Acal_5$-action, whose underlying coarse moduli space $B\Acal_5(5,2,2,2)$ is a quasi-projective variety with 2 connected components. By assigning an object of $\ \BB\Acal_5(5,2,2,2)$ its $\Acal_5$-quotient we get a degree 20 cover of $\ \MM_0(5,2,2,2)$ (resp. a degree 20 cover between coarse moduli spaces $B\Acal_5(5,2,2,2)\to \Mcal_0(5,2,2,2)$).}
\end{thm}
\begin{proof}
  {The Proposition \ref{g1g2g3g4} implies that the set $\Acal_5(5,2,2,2)$ has 20 elements which is divided into two $\Mod_{0,4}$-orbit. Then the Theorem \ref{finemod} and Theorem \ref{thm:orbifoldmoduli} implies the two results.}
\end{proof}
\begin{rmk}
  {The two components of the moduli space $B\Acal_5^M(5,2,2,2)$ resp. $B\Acal_5(5,2,2,2)$ are distinguished by the isomorphism class of the representation of $\Acal_5$ on $H^0(C,\omega_C)$ where $C$ is a smooth genus 10 $\Acal_5$-curve. And the two components are exchanged by the outer automorphism of $\Acal_5$.}
\end{rmk}

\begin{proof}[Proof of Theorem \ref{main1}]
   The fine moduli space $\Mcal_{10;4}^{\Acal_5}$ appearing in the statement of Theorem \ref{main1} is of course our ${B\mathcal{A}_5^M(5,2,2,2)}$. Then Theorems \ref{modspcursec}, \ref{modstkcur}, \ref{smfam}, \ref{universal} and Proposition \ref{deg} have established most parts of Theorem. The Proposition \ref{howCkhapp} implies $C_K$ is the only smooth $\Acal_5$-curve with non-trivial $\Acal_5$-automorphisms. Hence $\Scal_3$ action on $\Mcal_{10;4}^{\Acal_5}$ has only one irregular orbit.

We only need to prove that this orbit space $\bar{B}\mathcal{A}_5(5,2,2,2)={\bar{B}\mathcal{A}_5^M(5,2,2,2)}/\Scal_3$ is the Winger pencil. Let us first consider the smooth part ${B}\mathcal{A}_5(5,2,2,2)=B\mathcal{A}_5^M(5,2,2,2)/\Scal_3$. Using the construction in the proof of Theorem \ref{finemod} we can construct the "universal" family $\Cs^{\circ}\to B\Acal_5(5,2,2,2)\backslash[C_K]$ of smooth 10 $\Acal_5$-curves which missing $C_K$ only. Note that $B\Acal_5(5,2,2,2)\backslash[C_K]$ has dimension one. The univsersal property implies that there exist a map $\Bs^{\circ}\to B\Acal_5(5,2,2,2)\backslash[C_K]$ such that $\Ws^{\circ}$ is the pullback of $\Cs^{\circ}$ from this map.
  Now Theorem \ref{universal} shows that Winger pencil is locally universal at its smooth members and this pencil also has a dimension one base. The map $\Bs^{\circ}\to B\Acal_5(5,2,2,2)\backslash[C_K]$ must be an isomorphism. Hence by passing to the Deligne-Mumford compactification and adding $C_K$, we get the whole Winger pencil. This finishes the proof.
\end{proof}
\bibliographystyle{./spmpsci}

\bibliography{./Reference}

\end{document}